\documentclass[a4paper,11pt]{article}
\usepackage{amsmath, amsfonts, amscd, amssymb, amsthm, enumerate, xypic}

\def\bfB{\mathbf{B}}

\newcommand{\Hom}{\operatorname{Hom}}
\newcommand{\Mat}{\operatorname{M}}
\newcommand{\charac}{\chi}
\newcommand{\Mats}{\operatorname{S}}
\newcommand{\Mata}{\operatorname{A}}

\newcommand{\id}{\operatorname{id}}
\newcommand{\GL}{\operatorname{GL}}

\newcommand{\SL}{\operatorname{SL}}
\newcommand{\Ker}{\operatorname{Ker}}

\newcommand{\Vect}{\operatorname{span}}
\newcommand{\im}{\operatorname{Im}}

\newcommand{\tr}{\operatorname{tr}}

\newcommand{\Sp}{\operatorname{Sp}}
\newcommand{\rk}{\operatorname{rk}}

\newcommand{\botoplus}{\overset{\bot}{\oplus}}
\renewcommand{\setminus}{\smallsetminus}

\def\F{\mathbb{F}}

\def\R{\mathbb{R}}

\def\N{\mathbb{N}}
\def\Z{\mathbb{Z}}

\def\calL{\mathcal{L}}

\def\lcro{\mathopen{[\![}}
\def\rcro{\mathclose{]\!]}}

\theoremstyle{definition}
\newtheorem{Def}{Definition}[section]
\newtheorem{Not}[Def]{Notation}

\theoremstyle{plain}
\newtheorem{theo}{Theorem}[section]
\newtheorem{prop}[theo]{Proposition}
\newtheorem{cor}[theo]{Corollary}
\newtheorem{lemma}[theo]{Lemma}
\newtheorem{claim}{Claim}

\theoremstyle{plain}

\theoremstyle{remark}
\newtheorem{Rems}{Remarks}
\newtheorem{Rem}[Rems]{Remark}

\title{Products of involutions in symplectic groups over general fields (I)}
\author{Cl\'ement de Seguins Pazzis\footnote{Universit\'e de Versailles Saint-Quentin-en-Yvelines, Laboratoire de Math\'ematiques
de Versailles, 45 avenue des Etats-Unis, 78035 Versailles cedex, France}
\footnote{e-mail address: dsp.prof@gmail.com}}

\begin{document}

\thispagestyle{plain}

\maketitle

\begin{abstract}
Let $s$ be an $n$-dimensional symplectic form over an arbitrary field with characteristic not $2$, with $n>2$.

The simplicity of the group $\Sp(s)/\{\pm \id\}$ and the existence of a non-trivial involution in $\Sp(s)$ yield that every element of $\Sp(s)$ is a product of involutions.

Extending and improving recent results of Awa, de La Cruz, Ellers and Villa with the help of a completely new method,
we prove that if the underlying field is infinite, every element of $\Sp(s)$ is the product of four involutions if $n$ is a multiple of $4$, and of five involutions otherwise. The first part of this result is shown to be optimal for all multiples of $4$ and all fields, and is shown to fail for the fields with three elements and for $n=4$. Whether the second part of the result is optimal remains an open question.

Finite fields will be tackled in a subsequent article.
\end{abstract}




\section{Introduction}

\subsection{The problem}

The present article deals with the general problem of decomposing an element of a classical group into a product of involutions, i.e.\ of elements
of order $1$ or $2$.
An element of a group $G$ is called \textbf{$k$-reflectional} when it is the product of $k$ involutions of $G$.
In any symmetric group, every element is $2$-reflectional.
And the same holds in the orthogonal group of a quadratic form (see \cite{Wonenburger} for fields with characteristic not $2$, and \cite{Gow} for fields with characteristic $2$).
In the general linear group of a finite-dimensional vector space over a field, an element is $2$-reflectional if and only if it is conjugated to its inverse \cite{Djokovic,HoffmanPaige,Wonenburger};
an element is a product of involutions if and only if its determinant equals $\pm 1$, and in that case it is $4$-reflectional
\cite{Gustafsonetal} but not necessarily $3$-reflectional. And there is probably no reasonable characterization of the matrices that are $3$-reflectional
(see \cite{Liu} and \cite{dSPinvol3}, nevertheless).

Here, the classical group we are interested in is the somewhat neglected case of the symplectic group $\Sp(s)$ of a symplectic form $s$ on a vector space $V$ of dimension $n>0$ over a field $\F$ (of which we choose an algebraic closure $\overline{\F}$). Klaus Nielsen (unpublished, see a proof in \cite{dSPorthogonal}) characterized the $2$-reflectional elements of $\Sp(s)$; he showed in particular that $\Sp(s)$ always contains elements that are not $2$-reflectional provided that $\F$ does not have characteristic $2$.
This is in sharp contrast with the orthogonal groups, and it is not so surprising because the involutions in $\Sp(s)$ are not very natural objects in
geometric terms. Remember that the classical ``nice" generating subset of $\Sp(s)$ is the one of symplectic transvections, which play the role that reflections play in
orthogonal groups; hence it seems more natural to care about the elements of $\Sp(s)$ that are unipotent of index $2$, i.e.\ that satisfy
$(u-\id)^2=0$. And recently \cite{dSPunipotent} we have proved that if $\charac(\F)\neq 2$ every element of $\Sp(s)$ is the product
of three unipotent elements of index $2$, and of no less in general, and we have also characterized the elements of $\Sp(s)$ that are the product of
two unipotent elements of index $2$.
Moreover, with that in mind it is not surprising that fields of characteristic $2$ must be singled out, because
unipotent elements of index $2$ coincide with involutions in that case. From now on, we will systematically assume that
$\F$ does not have characteristic $2$.

If $n=2$, the only involutions in $\Sp(s)$ are $\pm \id$, they generate the subgroup $\{\pm \id\}$ and no more.
If $n>2$ however, there is a non-trivial involution in $\Sp(s)$, and since the quotient group $\Sp(s)/\{\pm \id\}$
is simple and $-v$ is an involution for every involution $v$, it follows that $\Sp(s)$ is generated by involutions.
However, this gives us no information on how to obtain decompositions with few factors.

Some further notation will be useful in order to state our main results.
Denote by $i_k(s)$ the set of all $k$-reflectional elements of $\Sp(s)$.
Clearly, $(i_k(s))_{k \geq 0}$ is a non-decreasing sequence of subsets of $\Sp(s)$ and we have just shown that its union equals $\Sp(s)$.
Moreover if for some $k \geq 0$ we have $i_k(s)=i_{k+1}(s)$, then clearly $i_{k+1}(s)=i_{k+2}(s)$ and hence $\Sp(s)=i_k(s)$.
So, in theory either $(i_k(s))_{k \geq 0}$ terminates and there is a least integer $k$ such that $i_k(s)=\Sp(s)$
and we denote this integer by $\ell_n(\F)$, or the sequence
$(i_k(s))_k$ is increasing and we set $\ell_n(\F):=+\infty$
(note that the group $\Sp(s)$ is determined up to isomorphism by the sole data of $\F$ and $n$, since symplectic forms over $\F$ are determined up to isometry
by the dimension of the underlying vector space).

We will actually prove that $\ell_n(\F)$ is finite whenever $n>2$, but this cannot be obtained only through abstract group-theoretic arguments.

When Nielsen characterized the $2$-reflectional elements in $\Sp(s)$, he also tried to solve the length problem for decompositions into involutions
but he never published his results (circa 1995). The problem was picked up much more recently by de La Cruz \cite{dLC}, who
was apparently unaware of Nielsen's work. De La Cruz restricted his scope to an algebraically closed field (his proof is formally about complex numbers, but the generalization to any algebraically closed field of characteristic other than $2$ is effortless)
and proved that $\ell_n(\F)\leq 4$ for all $n \geq 4$. We will show that this result is optimal. Unfortunately, symplectic groups over general fields are much more complicated to work with than over algebraically closed fields. Indeed, de La Cruz took huge advantage of the fact that the conjugacy class of an element $u$ of $\Sp(s)$ is entirely
described by the Jordan structure of $u$, which allows easy manipulations. In contrast, working with general fields is way more complicated because of the Wall invariants that are attached to conjugacy classes of symplectic transformations \cite{Wall}.

Following de La Cruz, Ellers and Villa \cite{EllersVilla} started tackling more general fields.
Using a very clever idea, they were able to prove that if $\F$ contains an element $\kappa$ such that $\kappa^2=-1$, then every element of $\Sp(s)$ is the product of two
quarter turns (i.e. elements $u$ such that $u^2=-\id$). Their argument relies upon a result of Wonenburger, who had proved, with the same technique
she employed for dealing with orthogonal groups, that every element of $\Sp(s)$ is the product of two \emph{skew-symplectic} involutions (and this result is non-trivial). If $n$ is a multiple of $4$, and still assuming that $t^2+1$ splits over $\F$, every quarter turn in $\Sp(s)$ is the product of two involutions (this is immediate by Nielsen's classification), and hence $\ell_n(\F) \leq 4$. Ellers and Villa went on to deduce that $\ell_n(\F) \leq 6$ if $n \geq 10$ and $n = 2$ mod 4
but their method failed to settle the case $n=6$. Finally, Awa and de La Cruz \cite{Awa} proved that if $n=4$ then $\ell_n(\R) \leq 4$.

It is our ambition to improve on all the results we have just stated by tackling all fields with characteristic not $2$.
In this first installment, we will mainly tackle infinite fields.
Finite fields constitute a much greater challenge, and they will be the topic of a subsequent article.
We also point to recent results \cite{Hou1,Hou2,Hou3} on products of commutators of involutions in symplectic groups, but we will not
try to generalize them as we feel that they are too field-specific.

\subsection{Main results}

We can now state our main results:

\begin{theo}[Main theorem]\label{theo:main}
Let $s$ be a symplectic form over an infinite field $\F$ of characteristic not $2$, with dimension $n$.
\begin{enumerate}[(a)]
\item If $n$ is a multiple of $4$, then every element of $\Sp(s)$ is the product of four involutions and of no less in general.

\item If $n=2$ mod $4$ and $n>2$, then every element of $\Sp(s)$ is the product of five involutions, and of no less than four in general.
\end{enumerate}
\end{theo}

Interestingly, it will be shown in a subsequent article that this result still holds over any finite field with characteristic not $2$
\emph{and more than $3$ elements}, but this requires a far more difficult analysis.
Fields with $3$ elements turn out to be very problematic in this decomposition problem.
As an example, in Section \ref{section:F3} we will produce an element of $\Sp_4(\F_3)$ that is not $4$-reflectional.

As far as $\ell_{4n+2}(\F)$ is concerned, it is still an open problem whether our results are optimal.
We suspect that the matrix
$$U:=\begin{bmatrix}
0_3 & -I_3 \\
I_3 & 0_3
\end{bmatrix}$$
is not $4$-reflectional in the symplectic group $\Sp_6(\R)$, but so far this special case has resisted our repeated efforts.

\begin{theo}\label{theo:nonincreasing}
For every infinite field $\F$, the sequence $(\ell_{4n+2}(\F))_{n \geq 1}$ is non-increasing.
\end{theo}

We conjecture that
$\ell_{4n+2}(\F)=4$ for all $n \geq 2$ (remember that de La Cruz proved this for every algebraically closed field of characteristic not $2$).
We suspect however that proving such a conjecture might be very hard, and in fact we have been entirely unable
to come up with any valid strategy for proving such a result. In any case, Theorem \ref{theo:nonincreasing} will certainly be useful in that prospect
(for example, if one could prove that $\ell_{10}(\F)=4$ for all infinite fields, then the question would be essentially settled).

On a heuristical note, it would not be entirely surprising to find values of $n \geq 2$ such that $\ell_{4n+2}(\F)=5$.
Indeed, in decomposing an automorphism into a product of two involutions, the greatest variety of results is generally obtained when
the eigenspaces of the factors have their dimension equal to half the dimension of the full space.
But in dimension $4n+2$ this situation cannot occur for symplectic involutions because the eigenspaces of an $s$-symplectic involution
must be $s$-regular, and hence must have their dimension even! So, there is a sort of ``shortage" of $s$-symplectic involutions in dimensions of the form
$4n+2$, and this shortage seems to account for the fundamental difference between the sequences $(\ell_{4n}(\F))_{n \geq 1}$ and $(\ell_{4n+2}(\F))_{n \geq 1}$.

\subsection{Structure of the article}

The remainder of the article is laid out as follows. In Section \ref{section:basics}, we will start with a limited survey of the classification of conjugacy classes in symplectic groups. Here, our techniques do not require a full understanding of them such
as the one obtained by Springer in \cite{Springer} and Wall in \cite{Wall}, so we will limit the discussion to some rough properties, as well as standard basic results.
We will of course recall Nielsen's theorem (unpublished by him, but reproved in \cite{dSPorthogonal}), which describes the products of two involutions,
and along the way we will recall a few standard results on symplectic transformations.

Then, in Section \ref{section:spacepullback} we will explain our main new technique to obtain decompositions into products of symplectic involutions.
The difficulty is to find a ``good" first factor, i.e.\ choose a symplectic involution $i$ wisely so that $iu$ has a shape that could enable us to decompose it into a product of few symplectic involutions. The idea is that, if possible we want $iu$ to stabilize a Lagrangian, or at least a very large totally $s$-singular subspace. So the idea will be to find such a subspace $W$ such that $u(W) \cap W = \{0\}$, and then to choose $i$ wisely to pull $u(W)$ back to $W$.
We will call this the \textbf{space-pullback technique}. In connection with this technique, the necessity comes of finding such ``nice" subspaces $W$, and
in most cases we will be able to find Lagrangians: this construction is explained in Section \ref{section:Lagrangianfind}.

From there, almost everything will be in place. With the space-pullback technique, we will be directly able to
obtain that both sequences $(\ell_{4n}(\F))_{n \geq 1}$ and $(\ell_{4n+2}(\F))_{n \geq 1}$ are non-increasing if $\F$ is infinite,
and we will also settle most cases of the $4$-dimensional situation. The $4$-dimensional situation will then be completed in Section
\ref{section:n=4}, thereby completing the proof of the first part of point (a) in Theorem \ref{theo:main}.

In Section \ref{section:n=6}, we consider the $6$-dimensional situation. This one requires a different method, which is reminiscent of the space-pullback technique, but the difference is that
we will try to find $i$ not as an involution but as the product of two involutions. And then we will find that $\ell_6(\F) \leq 5$,
which will conclude the proof of point (b) of Theorem \ref{theo:main}.

The last section is devoted to examples that show how optimal some of the previous results are. We will give a systematic example
of a symplectic transformation that is not $3$-reflectional (such an example was missing from the literature so far).
And we will see that over $\F_3$ there exist symplectic transformations in dimension $4$ that are not $4$-reflectional.

Throughout, we will only limit the discussion to infinite fields when it is absolutely necessary. In fact, most of the lemmas proved here hold for general
fields, and we will take advantage of this in the prospect of future work on finite fields.

\section{A review of symplectic transformations}\label{section:basics}

\subsection{Additional notation}

We denote by $\Mat_n(\F)$ the algebra of $n$-by-$n$ square matrices with entries in $\F$, by $\GL_n(\F)$
its group of invertible elements, and by $\SL_n(\F)$ its subgroup of all matrices with determinant $1$.

We also denote by $\Mats_n(\F)$ the linear subspace of $\Mat_n(\F)$ consisting of the symmetric matrices, and by
$\Mata_n(\F)$ the one consisting of the alternating matrices, i.e.\ the skew-symmetric matrices (those two notions are equivalent because we assume here
that $\F$ does not have characteristic $2$).

\subsection{Basics}\label{section:symplecticbases}

Let $K$ be an invertible skewsymmetric matrix of $\Mat_{2n}(\F)$.
A matrix $M \in \Mat_{2n}(\F)$ is called $K$-\textbf{symplectic} whenever $M^TKM=K$, i.e.\
$X \mapsto MX$ belongs to the symplectic group of the symplectic form $(X,Y) \mapsto X^T K Y$ on $\F^{2n}$.
If $K$ is the Gram matrix of a symplectic form $s$ in some basis $\bfB$, then the $K$-symplectic matrices are the matrices that represent the elements of
$\Sp(s)$ in $\bfB$.

\begin{Def}
An \textbf{s-pair} $(s,u)$ consists of a symplectic form $s$ on a vector space $V$ and of a symplectic transformation $u \in \Sp(s)$.
We will say that $V$ is its \textbf{underlying vector space}, and the dimension of $V$ is called the \textbf{dimension} of $(s,u)$.

Two s-pairs $(s,u)$ and $(s',u')$, with underlying vector spaces $V$ and $V'$, are called \textbf{isometric} if there exists an isometry
$\varphi : (V,s) \overset{\simeq}{\rightarrow} (V',s')$ such that $u'=\varphi \circ u \circ \varphi^{-1}$.
In that case, we note that $u$ is $k$-reflectional in $\Sp(s)$ if and only if $u'$ is $k$-reflectional in $\Sp(s')$.
\end{Def}

Let $(s,u)$ be an s-pair with underlying vector space $V$. We will systematically consider orthogonality with respect to $s$ unless stated otherwise.
Now, assume that we have a splitting $V=V_1 \botoplus V_2$ in which
$V_1$ and $V_2$ are stable under $u$ and $s$-orthogonal (so that they are $s$-regular). Denote by $u_1,u_2$ the respective endomorphisms of $V_1$ and $V_2$
induced by $u$, and by $s_1,s_2$ the respective symplectic forms induced by $s$ on $V_1$ and $V_2$.
Then $(s_1,u_1)$ and $(s_2,u_2)$ are s-pairs, and we write short $u=u_1 \botoplus u_2$ and $(s,u)=(s_1,u_1) \botoplus (s_2,u_2)$.
If $V \neq \{0\}$ and the only decompositions of the previous forms are such that $V_1=V$ or $V_2=V$, then we will say that the s-pair $(s,u)$
is \textbf{indecomposable} (this means that $u$ stabilizes no non-trivial $s$-regular subspace of $V$).

Conversely, assume that we have a symplectic form $s$ together with a decomposition $V=V_1 \botoplus V_2$, and denote by $s_1,s_2$ the resulting symplectic forms
on $V_1$ and $V_2$.
Let $u_1 \in \Sp(s_1)$ and $u_2 \in \Sp(s_2)$.
Assume that, for some $k \geq 1$, $u_1$ is  $k$-reflectional in $\Sp(s_1)$ and $u_2$ is $k$-reflectional in $\Sp(s_2)$.
We claim that $u:=u_1 \botoplus u_2$, which belongs to $\Sp(s)$, is $k$-reflectional in $\Sp(s)$.
Indeed, we can factorize $u_1=\prod_{l=1}^k i_1^{(l)}$ and $u_2=\prod_{l=1}^k i_2^{(l)}$
for symplectic involutions $i_1^{(1)}, \dots ,i_1^{(k)}$ in $\Sp(s_1)$ and
symplectic involutions $i_2^{(1)}, \dots ,i_2^{(k)}$ in $\Sp(s_2)$.
Then $u=\prod_{l=1}^k (i_1^{(l)} \botoplus i_2^{(l)})$, and the factors in this product are
involutions in $\Sp(s)$.

\begin{Def}
Let $s$ be a symplectic form on a vector space $V$.
A family $(e_1,\dots,e_n,f_1,\dots,f_n)$ of vectors of $V$ is called $s$-\textbf{symplectic}
if $s(e_i,f_j)=\delta_{i,j}$ for all $(i,j)\in \lcro 1,n\rcro^2$,
$s(f_j,e_i)=-\delta_{i,j}$ for all $(i,j)\in \lcro 1,n\rcro^2$, and $s$ maps all the other pairs in the family to $0$.

A \textbf{mixed $s$-symplectic} family of $V$ (with parameter $k$) is a family of vectors of the form
$(e_1,\dots,e_p,f_1,\dots,f_{2k},g_1,\dots,g_p)$ in which each $f_i$ is $s$-orthogonal to each one of $e_1,\dots,e_p,g_1,\dots,g_p$,
and both families $(e_1,\dots,e_p,g_1,\dots,g_p)$ and $(f_1,\dots,f_{2k})$ are $s$-symplectic.
\end{Def}

In any $s$-symplectic \emph{basis} $\bfB=(e_1,\dots,e_n,f_1,\dots,f_n)$ of $V$, the Gram matrix of $s$ equals
$$K_{2n}:=\begin{bmatrix}
0 & I_n \\
-I_n & 0
\end{bmatrix},$$
and the elements of $\Sp(s)$ are the endomorphisms $u$ of $V$ with matrix $M$ in $\bfB$ that belongs to the symplectic matrix group
$$\Sp_{2n}(\F)=\{M \in \Mat_{2n}(\F) : \; M^T K_{2n} M=K_{2n}\}.$$
In a mixed $s$-symplectic basis $(e_1,\dots,e_p,f_1,\dots,f_{2k},g_1,\dots,g_p)$ with parameter $k$,
the Gram matrix of $s$ equals
$$\begin{bmatrix}
0 & 0 & I_p \\
0 & K_{2k} & 0 \\
-I_p & 0 & 0
\end{bmatrix}.$$
Mixed symplectic bases are naturally useful in the following situation: suppose that we have a subspace $W$ that is totally $s$-singular
and stabilized by $u$. Then $u$ also stabilizes $W^{\bot}$, $s$ induces a symplectic form on $W^{\bot}/W$ which we denote by $\overline{s}$,
and the endomorphism $\overline{u}$ induced by $u$ on $W^{\bot}/W$ is $\overline{s}$-symplectic.
Then we take:
\begin{itemize}
\item a basis $(e_1,\dots,e_p)$ of $W$;
\item a direct factor $W'$ of $W^\bot$ in $V$, and then obtain a basis $(g_1,\dots,g_p)$ of $W'$ such that $(e_1,\dots,e_p,g_1,\dots,g_p)$ is $s$-symplectic;
\item and finally an $s$-symplectic basis $(f_1,\dots,f_{2k})$ of $(W+W')^{\bot}$.
\end{itemize}
The resulting family $(e_1,\dots,e_p,f_1,\dots,f_{2k},g_1,\dots,g_p)$ is then a mixed $s$-symplectic basis of $V$ with parameter $k$,
and the matrix of $u$ in that basis reads
$$\begin{bmatrix}
A & ? & ? \\
0 & B & ? \\
0 & 0 & ?
\end{bmatrix},$$
where $B$ represents, in the $\overline{s}$-symplectic basis $(\overline{f_1},\dots,\overline{f_{2k}})$ of $W^{\bot}/W$, the symplectic transformation $\overline{u}$.
And in particular $B \in \Sp_{2k}(\F)$.

\subsection{Polynomials and symplectic transformations}\label{section:polynomial}

Let $p \in \F[t]$ be a polynomial of degree $d>0$. We denote by
$$p^\sharp:=p(0)^{-1} t^d p(t^{-1})$$
its \textbf{reciprocal polynomial}. We say that $p$ is a \textbf{palindromial} whenever $p=p^\sharp$.
If $p$ is a palindromial of odd degree, then it has a root in $\{1,-1\}$,
whereas if $p$ is an irreducible palindromial distinct from $t\pm 1$ then $p$ has even degree and $p(0)=1$.
By factoring, it is easy to see that the monic palindromials are the products of irreducible palindromials of even degree
and of powers of $t+1$ and $t-1$.

If we have an automorphism $u$ of a vector space, and its invariant factors are $p_1,\dots,p_r$,
then the invariant factors of $u^{-1}$ are $p_1^\sharp,\dots,p_r^\sharp$, whence
$u$ is similar to its inverse if and only if each $p_k$ is a palindromial.

Noting that an $s$-symplectic transformation $u$ is conjugated in the general linear group to its inverse (because each endomorphism is similar to its transpose),
we get that its invariant factors are all palindromials.

Now, let $(s,u)$ be an s-pair and $p$ be an arbitrary non-zero polynomial.
Denote by $v^\star$ the $s$-adjoint of an endomorphism $v$ of the underlying vector space of $(s,u)$.
Then $p(u)^\star=p(u^{-1})=\lambda \, u^k p^\sharp (u)=\lambda\, p^\sharp(u) u^k$
for some $\lambda \in \F \setminus \{0\}$ and some $k \in \Z$, leading to $\im p(u)^\star=\im p^\sharp (u)$ and $\Ker p(u)^\star=\Ker p^\sharp (u)$,
which further leads to $(\Ker p(u))^\bot=\im p^\sharp (u)$ and $(\im p(u))^\bot=\Ker p^\sharp (u)$.
If $p$ is relatively prime with $p^\sharp$, then B\'ezout's theorem shows that $\Ker p(u) \subseteq \im p^\sharp(u)$ and we deduce that $\Ker p(u)$ is totally $s$-singular.

\subsection{Lagrangians and symplectic extensions}

Let $(V,s)$ be a symplectic space of dimension $2n$. Remember that a \textbf{Lagrangian} of $(V,s)$ is a linear subspace
$\mathcal{L}$ of $V$ that is totally $s$-singular and of dimension $n$.
Two Lagrangians $\mathcal{L}$ and $\mathcal{L'}$ are called \textbf{transverse} whenever $\mathcal{L} \cap \mathcal{L}'=\{0\}$.
Assume that we have two such Lagrangians $\mathcal{L}$ and $\mathcal{L'}$, and another Lagrangian $\mathcal{L}''$ that is transverse to $\mathcal{L}$.
Let $(e_1,\dots,e_n)$ be a basis of $\mathcal{L}$. Then there are unique bases $(f_1,\dots,f_n)$ and $(g_1,\dots,g_n)$, of $\mathcal{L}'$
and $\mathcal{L}''$ respectively, such that $\bfB':=(e_1,\dots,e_n,f_1,\dots,f_n)$ and $\bfB'':=(e_1,\dots,e_n,g_1,\dots,g_n)$ are $s$-symplectic
bases of $V$. And one checks that the matrix of $\bfB''$ in $\bfB'$ is of the form $\begin{bmatrix}
I_n & S \\
0_n & I_n
\end{bmatrix}$ for a \emph{symmetric} matrix $S \in \Mats_n(\F)$. Conversely, given $S \in \Mats_n(\F)$, the basis of $V$
whose matrix in $\bfB'$ equals $P:=\begin{bmatrix}
I_n & S \\
0_n & I_n
\end{bmatrix}$ is of the form $(e_1,\dots,e_n,h_1,\dots,h_n)$, one checks that $P$ is $K_{2n}$-symplectic and hence
$(e_1,\dots,e_n,h_1,\dots,h_n)$ is a symplectic basis of $(V,s)$, yielding in particular that $\Vect(h_1,\dots,h_n)$ is a Lagrangian that is transverse
to $\mathcal{L}$.

Now, let $\mathcal{L}$ and $\mathcal{L}'$ be transverse Lagrangians of $(V,s)$.
Denote by $\mathcal{L}^\star:=\Hom(\mathcal{L},\F)$ the dual vector space of $\mathcal{L}$. For an endomorphism $f$ of $\mathcal{L}$, denote by $f^t$ its transposed endomorphism of $\mathcal{L}^\star$, defined as $\varphi \in \calL^\star \mapsto \varphi \circ f \in \calL^\star$.

Let $v \in \GL(\mathcal{L})$. The symplectic form $s$ induces an isomorphism $\varphi : x \in \mathcal{L}' \mapsto s(-,x) \in \mathcal{L}^\star$, and we can consider the automorphism $v':=\varphi^{-1} \circ (v^{-1})^t \circ \varphi$ of $\mathcal{L'}$. One checks that the direct sum $v \oplus v'$ is $s$-symplectic, and obviously it stabilizes $\mathcal{L}$ and $\mathcal{L'}$: we call it the \textbf{symplectic extension} of $v$ to the Lagrangian $\mathcal{L}'$, denoted by $s_{\mathcal{L}'}(v)$. And conversely, if an $s$-symplectic transformation $u$ stabilizes both $\mathcal{L}$ and $\mathcal{L'}$,
then it is the symplectic extension of $u_{\mathcal{L}}$ to $\mathcal{L}'$.
Besides it is obvious from the definition that
$$v \in \GL(\mathcal{L}) \mapsto s_{\mathcal{L'}}(v) \in \Sp(s)$$
is a group homomorphism.
As a consequence, we immediately get:

\begin{lemma}\label{lemma:extensiontoproduct}
Let $(s,u)$ be a symplectic pair. Assume that $u$ is a symplectic extension of an automorphism $v$ of a Lagrangian
$\mathcal{L}$, and assume that $v$ is $k$-reflectional in $\GL(\mathcal{L})$. Then
$u$ is $k$-reflectional in $\Sp(s)$.
\end{lemma}

Symplectic extensions have remarkable properties. For example, if $u$ is the symplectic extension of $v \in \GL(\mathcal{L})$ to the
Lagrangian $\mathcal{L'}$, then for every polynomial $p \in \F[t]$, the quadratic form $x \mapsto s(x,p(u)[x])$
vanishes everywhere on $\mathcal{L}$ and on $\mathcal{L'}$, and classically this shows that its regular part is hyperbolic.

In matrix terms, if we have a symplectic basis $\bfB'$ of $V$ that is adapted to the decomposition $V=\mathcal{L}\oplus \mathcal{L}'$,
and $u$ stabilizes both $\mathcal{L}$ and $\mathcal{L}'$, then the matrix of $u$ in that basis reads
$$\begin{bmatrix}
A & 0 \\
0 & (A^T)^{-1}
\end{bmatrix}.$$

\begin{Not}
For an invertible matrix $A \in \GL_n(\F)$, we set
$$A^\sharp:=(A^T)^{-1}.$$
\end{Not}

\begin{Rem}\label{remark:stablesingular}
Here is a generalization that will be useful in later parts of the article. Suppose that we have a totally $s$-singular
subspace $W$ that is stable under $u \in \Sp(s)$.
Then $\varphi : \overline{x} \in V/W^\bot \mapsto s(-,x) \in W^\star$ is a vector space isomorphism. The endomorphism $u$ induces an endomorphism $\overline{u}$ of $V/W^\bot$,
and one checks that $\overline{u}=\varphi^{-1} \circ ((u_W)^t)^{-1} \circ \varphi$.
In matrix terms, this means that if we have a basis $(e_1,\dots,e_p)$ of $W$, and then we extend it to a mixed $s$-symplectic basis
$(e_1,\dots,e_p,f_1,\dots,f_{2k},g_1,\dots,g_p)$ with parameter $k$,
the matrix of $u$ in that basis reads
$$\begin{bmatrix}
A & ? & ? \\
0 & B & ? \\
0 & 0 & A^\sharp
\end{bmatrix},$$
where the matrix $A$ represents $u_W$.
\end{Rem}

In specific situations, symplectic extensions are easy to recognize:

\begin{lemma}\label{lemma:extensionrecog}
Let $(s,u)$ be an s-pair. Assume that the characteristic polynomial of $u$ reads $pp^\sharp$ where $p$ is monic and relatively prime with $p^\sharp$.
Then $u$ is a symplectic extension of an automorphism with characteristic polynomial $p$.
\end{lemma}

\begin{proof}
Because of the primality assumptions, we have seen in Section \ref{section:polynomial} that $\Ker p(u)$ and $\Ker p^\sharp(u)$ are totally s-singular.
Besides $V=\Ker p(u) \oplus \Ker p^\sharp(u)$, and hence $\mathcal{L}:=\Ker p(u)$ and $\mathcal{L}':=\Ker p^\sharp(u)$
are transverse Lagrangians. Both are stable under $u$, whence $u$ is the symplectic extension to $\mathcal{L}'$ of its restriction $v$ to $\mathcal{L}$.
And clearly the characteristic polynomial of $v$ is $p$.
\end{proof}

\subsection{Indecomposables, and Nielsen's theorem}

Now, we will recall a rough description of the indecomposable s-pairs.
If an s-pair $(s,u)$ is indecomposable, then:
\begin{itemize}
\item Either $u$ is cyclic with minimal polynomial of the form $p^n$, where $n \geq 1$ and $p$ is a monic irreducible palindromial of even degree; we call this a \textbf{type I cell};
\item Or $u$ is cyclic with minimal polynomial of the form $q^n(q^\sharp)^n$, where $n \geq 1$ and $q$ is a monic irreducible polynomial such that $q^\sharp \neq q$; in that case
$u$ is a symplectic extension of a cyclic automorphism with minimal polynomial $q^n$; we call this a \textbf{type II cell};
\item Or $u$ is cyclic with minimal polynomial of the form $(t-\eta)^{2n}$ for some $\eta=\pm 1$ and some integer $n \geq 1$;
we call this a \textbf{type III cell};
\item Or $u$ is a symplectic extension of a cyclic automorphism with minimal polynomial $(t-\eta)^{2n+1}$ for some $\eta=\pm 1$ and some integer $n \geq 0$;
we call this a \textbf{type IV cell}.
\end{itemize}
Note that, in sharp contrast with the first three cases, the last case has the endomorphism $u$ with two invariants factors, both equal to $(t-\eta)^{2n+1}$.

Wall's theorem \cite{Wall} goes far beyond the above as it explains when two decompositions correspond to conjugate symplectic transformations, but fortunately we will not need such a precise understanding of the situation.

What we will need though is Nielsen's theorem on products of two involutions, which we shorten as follows:

\begin{theo}[Nielsen's theorem, see \cite{dSPorthogonal} for a proof]
Let $(s,u)$ be an s-pair.
For $u$ to be the product of two involutions of $\Sp(s)$, it is necessary and sufficient that $u$ be a symplectic extension
of an automorphism which is similar to its inverse (or, equivalently, of an automorphism which is the product of two involutions in the corresponding general linear group).
\end{theo}

Note that the converse implication is obvious from Lemma \ref{lemma:extensiontoproduct}.

In matrix terms, Nielsen's theorem means that a matrix $M$ of $\Sp_{2n}(\F)$ is $2$-reflectional in $\Sp_{2n}(\F)$ if and only if, for some $A \in \GL_n(\F)$
that is similar to its inverse, $M$ is symplectically similar (i.e.\ conjugated through an element of $\Sp_{2n}(\F)$) to $\begin{bmatrix}
A & 0 \\
0 & A^\sharp
\end{bmatrix}$.

It follows in particular that the invariant factors of such a $2$-reflectional element $u$ come in pairs, but this condition is only necessary, not sufficient, because
of the Wall invariants. It is however sufficient in the triangularizable case when there is no eigenvalue in $\{\pm 1\}$, a case we will use at times:

\begin{cor}\label{cor:Nielsencor}
Let $(s,u)$ be an s-pair. Assume that $u$ is triangularizable with no eigenvalue in $\{\pm 1\}$.
Assume furthermore that, for each eigenvalue $\lambda$ of $u$ and each integer $k \geq 1$, there is an even number of Jordan cells
of $u$ for the eigenvalue $\lambda$ and with size $k$. Then $u$ is $2$-reflectional in $\Sp(s)$.
\end{cor}

The proof is based upon a part of Wall's classification that does not require the so-called Hermitian invariants, and which we recall now:

\begin{prop}\label{prop:caracconjsansWall}
Let $(s,u)$ and $(s',u')$ be s-pairs. Assume that none of the invariants factors of $u$ and $u'$ is a multiple of an irreducible palindromial.
For the pairs $(s,u)$ and $(s',u')$ to be isometric, it is then necessary and sufficient that $u$ be similar to $u'$.
\end{prop}

\begin{proof}[Proof of Corollary \ref{cor:Nielsencor}]
For $\lambda \in \F \setminus \{0\}$, denote by $n_{\lambda,k}(u)$ the number of Jordan cells of size $k$ of $u$ for the eigenvalue $\lambda$.

We can choose a Lagrangian $\calL$ of the underlying vector space $V$ of $u$, together with a triangularizable automorphism $v$ of $\calL$ such that
$n_{\lambda,k}(v)=\frac{1}{2}\,n_{\lambda,k}(u)$ for all $\lambda \in \F \setminus \{0\}$ and all $k \geq 1$.
Let us choose a symplectic extension $u'$ of $v$. Then $u'$ is triangularizable and, for all $k \geq 1$ and all $\lambda \in \F \setminus \{0\}$,
$$n_{\lambda,k}(u')=n_{\lambda,k}(v)+n_{\lambda,k}((v^t)^{-1})=n_{\lambda,k}(v)+n_{\lambda^{-1},k}(v)
=\frac{n_{\lambda,k}(u)+n_{\lambda^{-1},k}(u)}{2}=n_{\lambda,k}(u).$$
From Proposition \ref{prop:caracconjsansWall}, we deduce that $(s,u)$ and $(s,u')$ are isometric.
Finally, $v$ is similar to $v^{-1}$ because $v$ is triangularizable and, for all $k \geq 1$ and for all $\lambda \in \F \setminus \{0\}$,
$$n_{\lambda,k}(v)=\frac{n_{\lambda,k}(u)}{2}=\frac{n_{\lambda^{-1},k}(u)}{2}=n_{\lambda^{-1},k}(v)=n_{\lambda,k}(v^{-1}).$$
By Nielsen's theorem, we conclude that $u'$ is $2$-reflectional, and hence so is $u$.
\end{proof}

\subsection{Extra lemmas}

Before we can get to our problem, an extra lemma is required.

\begin{lemma}\label{lemma:split}
Let $(s,u)$ be an s-pair, with underlying vector space $V$. Assume that we have a totally $s$-singular subspace $W$ of $V$
that is stable under $u$. Denote by $p$ the characteristic polynomial of $u_{W}$, by
$\overline{s}$ the symplectic form on $W^{\bot}/W$ induced by $s$, and by
$\overline{u}$ the $\overline{s}$-symplectic transformation of $W^{\bot}/W$ induced by $u$.
Assume finally that the characteristic polynomial of $\overline{u}$ is relatively prime with $pp^\sharp$.
Then there is a splitting $(s,u)=(s_0,u_0) \botoplus (s_1,u_1)$ in which:
\begin{enumerate}[(i)]
\item The underlying vector space $V_0$ of $(s_0,u_0)$ includes $W$ as a Lagrangian that is stable under $u_0$ (with induced endomorphism equal to $u_W$).
\item The s-pair $(s_1,u_1)$ is isometric to $(\overline{s},\overline{u})$.
\end{enumerate}
\end{lemma}

\begin{proof}
Let us take a mixed symplectic basis $(e_1,\dots,e_l,f_1,\dots,f_{2k},g_1,\dots,g_l)$ of $(V,s)$
such that $(e_1,\dots,e_l)$ is a basis of $W$ (see Section \ref{section:symplecticbases}).
The matrix $M$ of $u$ in that basis reads
$$\begin{bmatrix}
A & [?] & [?] \\
[0] & B & [?] \\
0_l & [0] & A^\sharp
\end{bmatrix}$$
where $A$ represents $u_W$ and $B$ represents $\overline{u}$. Hence $\chi_A=p$ and $\chi_{A^\sharp}=p^\sharp$.

Set $q:=\chi_B$. Since $q$ is relatively prime with $pp^\sharp$, B\'ezout's theorem shows that $q(A)$ and $q(A^\sharp)$ are invertible, whereas
the Cayley-Hamilton theorem shows that $q(B)=0$. Computing block-wise yields $\rk q(M)=2l$, and hence
$\dim \Ker q(u)=2k$. Likewise, $\dim \Ker q(u_{W^\bot})=2k$. Since obviously $\Ker q(u_{W^\bot}) \subseteq \Ker q(u)$, we
deduce that $\Ker q(u_{W^\bot})=\Ker q(u)$ and that both spaces have dimension $\dim W^{\bot}-\dim W$.

Next, by combining the Cayley-Hamilton theorem with the kernel decomposition theorem, we find
$W \subseteq \Ker p(u_{W^\bot})$ and
$$W^{\bot}=\Ker p(u_{W^\bot}) \oplus \Ker q(u_{W^\bot}).$$
Since $\dim \Ker q(u_{W^\bot})=\dim W^{\bot}-\dim W$, we find $\dim \Ker p(u_{W^\bot})=\dim W$ and hence
$W=\Ker p(u_{W^\bot})$. In turn, this shows that
\begin{equation}\label{equation:split}
W^{\bot}=W \oplus \Ker q(u).
\end{equation}
Finally, we split $V=\Ker (pp^\sharp)(u) \botoplus \Ker q(u)$, where the orthogonality comes from the fact that
$pp^\sharp$ and $q$ are relatively prime palindromials. We set $V_0:=\Ker (pp^\sharp)(u)$, $V_1:=\Ker q(u)$, and
we denote by $u_0$ and $u_1$ the respective restrictions of $u$ to $V_0$ and $V_1$, and by $s_0$ and $s_1$
the respective restrictions of $s$ to $(V_0)^2$ and $(V_1)^2$.

Identity \eqref{equation:split} shows that $\Ker q(u)$ is projected isometrically onto $W^{\bot}/W$ with respect to $s$, and hence
$(\overline{s},\overline{u})$ is isometric to $(s_1,u_1)$.
In turn, as $\overline{u}$ is similar to $u_1$ we derive that $\chi_{u_1}=q$, and hence $\chi_{u_0}=q^{-1} \chi_u=pp^\sharp$.

Finally, $\dim V_0=2\dim W$ and $W \subseteq V_0$, whence the totally $s$-singular subspace $W$ is a Lagrangian of $(V_0,s_0)$.
\end{proof}

In matrix terms, we have the following interpretation:
let $u$ be represented in some mixed symplectic basis $(e_1,\dots,e_l,f_1,\dots,f_{2k},g_1,\dots,g_l)$ by a matrix
$$\begin{bmatrix}
A & [?] & [?] \\
[0] & B & [?] \\
0_l & [0] & A^\sharp
\end{bmatrix}$$
where $B \in \Sp_{2k}(\F)$, and assume that $\chi_B$ is relatively prime with $\chi_A \chi_A^\sharp$.
Then $u$ splits into an orthogonal direct sum $u_1\botoplus u_2$ where:
\begin{itemize}
\item $u_2$ is represented by $B$ in some symplectic basis;
\item $u_1$ is represented, in some symplectic basis, by a matrix of the form $\begin{bmatrix}
A & [?] \\
0_l & A^\sharp
\end{bmatrix}$.
\end{itemize}
The difficulty here lies in the lack of control on the upper-right block of the latter matrix, so in practice we will seek to be in situations
where $\chi_A$ is relatively prime with its reciprocal polynomial (and not only with $\chi_B$), so that we can use Lemma \ref{lemma:extensionrecog}:

\begin{cor}\label{lemma:split2}
Let $(s,u)$ be an s-pair, with underlying vector space $V$. Assume that we have a totally $s$-singular subspace $W$ of $V$
that is stable under $u$. Denote by $p$ the characteristic polynomial of $u_{W}$, by
$\overline{s}$ the symplectic form on $W^{\bot}/W$ induced by $s$, and by
$\overline{u}$ the $\overline{s}$-symplectic transformation of $W^{\bot}/W$ induced by $u$.
Assume finally that the characteristic polynomial of $\overline{u}$ is relatively prime with $pp^\sharp$, and that $p$ is relatively prime with $p^\sharp$. Then
there is a splitting $(s,u)=(s_0,u_0) \botoplus (s_1,u_1)$ in which:
\begin{enumerate}[(i)]
\item $u_0$ is a symplectic extension of $u_W$;
\item The s-pair $(s_1,u_1)$ is isometric to $(\overline{s},\overline{u})$.
\end{enumerate}
\end{cor}

\section{The space-pullback technique}\label{section:spacepullback}

Let $u \in \Sp(s)$. We are trying to find involutions $i_1,\dots,i_k$, as few as possible, such that $i_k \circ \cdots \circ i_1 \circ s=\id$.
The key is to do a very wise choice of the first involution $i_1$, and here we will explain a basic technique to make such a choice.

We will start by laying out the idea in purely matrix terms, but for the concrete applications it will be crucial to understand the technique
in geometric terms.

\subsection{Matrix formulation of the space-pullback technique}

Let $n,p$ be positive integers such that $n \geq 2p>0$. We consider the invertible alternating matrix
$$K:=\begin{bmatrix}
0_{2p} & [0] & I_{2p} \\
[0] & K_{2n-4p} & [0] \\
-I_{2p} & [0] & 0_{2p}
\end{bmatrix} \in \GL_{2n}(\F).$$
Let us take a $K$-symplectic matrix of the form
$$M=\begin{bmatrix}
0_{2p} & [0] & [?] \\
[0] & [?] & [?] \\
S & [?] & [?]
\end{bmatrix} \quad \text{with $S \in \GL_{2p}(\F)$.}$$
A quick computation shows that
$$M=\begin{bmatrix}
0_{2p} & [0] & [?] \\
[0] & N & [?] \\
S & [?] & [?]
\end{bmatrix} \quad \text{for some symplectic matrix $N \in \Sp_{2n-4p}(\F)$.}$$
Now, let us take an arbitrary \emph{alternating} $A \in \GL_{2p}(\F) \cap \Mata_{2p}(\F)$
and an arbitrary \emph{involutory} $A_1 \in \Sp_{2n-4p}(\F)$. One checks that
$$\widetilde{A}:=\begin{bmatrix}
0_{2p} & [0] & A^{-1} \\
[0] & A_1 & [0] \\
A & [0] & 0_{2p}
\end{bmatrix}\in \Mat_{2n}(\F)$$
is a $K$-symplectic involution.
Moreover,
$$\widetilde{A}M=\begin{bmatrix}
A^{-1}S & [?] & [?] \\
[0] & A_1 N & [?] \\
0_{2p} & [0] & [?]
\end{bmatrix}.$$
Assume now that $A^{-1}S$ has no common eigenvalue in $\overline{\F}$ with its inverse, and that $A_1N$ and $A^{-1}S$
have no common eigenvalue in $\overline{\F}$. Then Corollary \ref{lemma:split2} shows that every symplectic transformation that is represented by $\widetilde{A}M$
in a mixed symplectic basis splits into an orthogonal direct sum of two symplectic transformation $u_1$ and $u_2$ such that:
\begin{itemize}
\item $u_1$ is a symplectic extension of an automorphism that is represented by the matrix $A^{-1}S$;
\item $u_2$ is represented by $A_1N$ in some symplectic basis.
\end{itemize}
If we know that both $u_1$ and $u_2$ are $r$-reflectional, then
$\widetilde{A}M$ is the product of $r$ $K$-symplectic involutory matrices, to the effect that
$M=\widetilde{A}(\widetilde{A}M)$ is the product of $r+1$ $K$-symplectic involutory matrices.

Now, say that we have a decomposition $N=S_1S_2\cdots S_r$ into the product of
$r$ symplectic involutory matrices. Then a natural choice for $A_1$ is to take $A_1:=S_1$,
and hence, three important properties must be required of $A$ if we want this technique to succeed:
\begin{itemize}
\item That any symplectic extension of any automorphism that is represented by $A^{-1}S$ be $(r-1)$-reflectional;
\item That the characteristic polynomial of $A^{-1}S$ have no common root in $\overline{\F}$ with its reciprocal polynomial;
\item That the characteristic polynomial of $S_2 \cdots S_r$ have no common root in $\overline{\F}$ with the one of $A^{-1}S$.
\end{itemize}

\subsection{Geometric formulation of the space-pullback technique}

We will now explain the underlying geometry behind the previous block-matrix technique.

\begin{Not}
Let $(s,u)$ be an s-pair, with underlying vector space $V$. We consider the bilinear mapping
$$s_u : \begin{cases}
V^2 & \longrightarrow \F \\
(x,y) & \longmapsto s(x,u(y)),
\end{cases}$$
and for any totally $s$-singular subspace $W$ of $V$ we denote by $s_{u,W}$ the restriction of $s_u$ to $W^2$.
\end{Not}

Beware that $s_u$ is neither symmetric nor skewsymmetric in general. In fact, its symmetric and skew-symmetric parts are, respectively,
$$(x,y) \mapsto \frac{1}{2}\,s\bigl(x,(u-u^{-1})(y)\bigr) \quad \text{and} \quad
(x,y) \mapsto \frac{1}{2}\,s\bigl(x,(u+u^{-1})(y)\bigr).$$

\begin{Rem}\label{remark:nondeg}
If $s_{u,W}$ is nondegenerate then automatically $W \cap u(W)=\{0\}$.
Indeed, note that $u(W)$ is totally $s$-singular, whence if we take $x \in W \cap u(W)$, then it is clear that
$s(x,y)=0$ for all $y \in u(W)$, and hence $x$ is in the left-radical of $s_{u,W}$.

Moreover, the converse statement holds whenever $W$ is a Lagrangian! Assume indeed that $W$ is a Lagrangian
and that $W \cap u(W)=\{0\}$. Then $u(W)$ is a Lagrangian that is transverse to $W$ and clearly
$s_{u,W}$ is nondegenerate.
\end{Rem}

Let us look more closely at the action of symplectic involutions on totally $s$-singular subspaces.
Our first observation is that if $a$ is an involution in $\Sp(s)$, then $a$ is $s$-selfadjoint, to the effect that
$s_a$ is an alternating bilinear form.

Now, let $W_1$ and $W_2$ be totally $s$-singular subspaces of $W$, with the same dimension $d$.
Assume furthermore that $W_1$ and $W_2$ are \textbf{$s$-paired}, meaning that the bilinear mapping
$(x,y) \in W_1 \times W_2 \mapsto s(x,y)$ is nondegenerate on both sides.
Note in particular that this requires $W_1 \cap W_2=\{0\}$.

Let $s'$ be a symplectic form on $W_1$.
There is a unique vector space isomorphism $a : W_1 \overset{\simeq}{\rightarrow} W_2$ such that
$s'(x,y)=s(x,a(y))$ for all $(x,y)\in (W_1)^2$.
We extend $a$ to a linear isomorphism $\widetilde{a} : W_1 \oplus W_2 \overset{\simeq}{\longrightarrow} W_1 \oplus W_2$
by taking $\widetilde{a}(y):=a^{-1}(y)$ for all $y \in W_2$, making $\widetilde{a}$ an involution of the vector space $W_1 \oplus W_2$.
And then we check that $\widetilde{a}$ is actually an $s$-symplectic transformation.
Let indeed $x_1,x'_1$ belong to $W_1$ and $x_2,x'_2$ belong to $W_2$. Then
\begin{align*}
s\bigl(\widetilde{a}(x_1+x_2),\widetilde{a}(x'_1+x'_2)\bigr)
& =s\bigl(a(x_1)+a^{-1}(x_2),a(x'_1)+a^{-1}(x'_2)\bigr) \\
& =s\bigl(a^{-1}(x_2),a(x'_1)\bigr)+s\bigl(a(x_1),a^{-1}(x'_2)\bigr) \\
& =s\bigl(a^{-1}(x_2),a(x'_1)\bigr)-s\bigl(a^{-1}(x'_2),a(x_1)\bigr) \\
& =s'\bigl(a^{-1}(x_2),x'_1\bigr)-s'\bigl(a^{-1}(x'_2),x_1\bigr) \\
& =s'\bigl(x_1,a^{-1}(x'_2)\bigr)-s'\bigl(x'_1,a^{-1}(x_2)\bigr) \\
& =s(x_1,x'_2)-s(x'_1,x_2) \\
& =s(x_1+x_2,x'_1+x'_2).
\end{align*}

Now, say that we start from $u \in \Sp(s)$ and that we have found a totally $s$-singular subspace
$W$ such that the bilinear form $s_{u,W}$ is nondegenerate.
Note by Remark \ref{remark:nondeg} that $W \cap u(W)=\{0\}$.
Let then $b$ be an arbitrary symplectic form on $W$.
There is a unique $v \in \GL(W)$ such that
$$\forall (x,y)\in W^2, \; s(x,u(y))=b(x,v(y)).$$
Now, define $i$ as the symplectic involution of $W \oplus u(W)$ that maps $W$ to $u(W)$ and such that
$s(x,i(y))=b(x,y)$ for all $(x,y) \in W^2$.
Let us extend $i$ into a symplectic involution of $V$, still denoted by $i$, by taking an arbitrary symplectic involution of $(W\oplus u(W))^{\bot}$
(e.g. the identity, but it is crucial that we leave other possibilities open). We shall say that the restriction of $i$ to $(W\oplus u(W))^{\bot}$
is the \textbf{residual involution} associated with $i$ and $W$ (with respect to $u$).

Then, $iu$ stabilizes $W$ and
$$\forall (x,y)\in W^2, \; s(x,u(y))=s(x,i(iu(y)))=b(x,iu(y)),$$
and it follows that $(iu)_{W}=v$.
Hence the term ``space-pullback" as $W$, which was ``pushed" to $u(W)$ by $u$, is pulled back to itself by $i$.

It follows that $iu$ induces a symplectic transformation $w$ of $W^{\bot}/W$. It will be crucial to observe that the
type of this symplectic transformation depends only on $u$ and on the choice of the residual involution associated with $i$ and $W$, but not on the specific choice of $b$.
To start with, note that $u$ maps $W$ into $u(W)$ and hence $W^{\bot}$ into $u(W)^{\bot}$.
In particular $u$ maps $(W \oplus u(W))^{\bot}$ into $u(W)^{\bot}=u(W) \oplus  (W \oplus u(W))^{\bot}$.
Denoting by $\pi$ the projection of $u(W) \oplus  (W \oplus u(W))^{\bot}$ onto $(W \oplus u(W))^{\bot}$ along $u(W)$,
we shall say that $x \in (W \oplus u(W))^{\bot} \mapsto \pi(u(x))$ is the \textbf{residual endomorphism} of
$(W \oplus u(W))^{\bot}$ associated with $u$ and $W$. It is easily seen that this endomorphism is $s$-symplectic:
let indeed $x$ and $y$ belong to $(W \oplus u(W))^{\bot}$, and split $u(x)=x'+x''$ and $u(y)=y'+y''$ with $x',y'$
in $u(W)$ and $x'',y''$ in $(W \oplus u(W))^{\bot}$. Then as $u(W)$ is totally $s$-singular and orthogonal to
$(W \oplus u(W))^{\bot}$, we find
$$s(x,y)=s(u(x),u(y))=s(x'',y'')=s\bigl(\pi(u(x)),\pi(u(y))\bigr).$$
Besides
$$\forall x \in (W \oplus u(W))^{\bot}, \quad (iu)(x)=\underbrace{i(\pi(u(x)))}_{\in (W \oplus u(W))^{\bot}}+\underbrace{i(u(x)-\pi(u(x)))}_{\in W}.$$
Hence, if we denote by $\pi'$ the projection of $W \oplus  (W \oplus u(W))^{\bot}$ onto $(W \oplus u(W))^{\bot}$ along $W$,
it turns out that $u' : x \in (W \oplus u(W))^{\bot} \mapsto \pi'((iu)(x))$ is the composite of the
residual endomorphism of $(W \oplus u(W))^{\bot}$ associated with $u$ with the residual involution of $(W \oplus u(W))^{\bot}$ associated with $i$.

\vskip 3mm

Let us sum up: to apply the space-pullback technique, we start from a totally $s$-singular subspace $W$ such that $s_{u,W}$ is nondegenerate, and
we choose a symplectic form $b$ on $s_{u,W}$ and a symplectic involution $i'$ of $(W\oplus u(W))^\bot$. This creates an involution $i \in \Sp(s)$
such that:
\begin{enumerate}[(i)]
\item $iu$ stabilizes $W$ and the restriction  $v$ of $iu$ to $W$ satisfies $b(x,v(y))=s_{u,W}(x,y)$ for all $(x,y) \in W^2$;
\item The symplectic transformation of $W^\bot/W$ induced by $iu$ is symplectically similar to
$i'u'$, where $u'$ stands for the residual symplectic transformation of $(W \oplus u(W))^{\bot}$ associated with $u$ and $W$.
\end{enumerate}
In practice, we will choose $W$ so that $s_{u,W}$ is symmetric. In that case, we have a one-to-one correspondence between
symplectic forms on $W$ and $s_{u,W}$-skewselfadjoint automorphisms, which takes the $s_{u,W}$-skewselfadjoint automorphism
$w$ to the symplectic form $(x,y) \mapsto s_{u,W}(x,w^{-1}(y))$, and if we choose the symplectic form $b$ that corresponds to $w$, then
for any symplectic involution $i \in \Sp(s)$ that is constructed thanks to the space-pullback technique applied to the triple $(u,W,b)$,
the automorphism of $W$ induced by $iu$ equals $w$.

We finish by connecting the previous geometric formulation with the matrix formulation.
Let us consider a mixed symplectic basis $\bfB=(e_1,\dots,e_{2l},f_1,\dots,f_{2k},g_1,\dots,g_{2l})$ of $V$
such that $(e_1,\dots,e_{2l})$ and $(g_1,\dots,g_{2l})$ are respective bases of $W$ and $u(W)$.
Denote by $A$ the Gram matrix of $b$ in $(e_1,\dots,e_{2p})$, by $S$ the Gram matrix of $s_{u,W}$ in $(e_1,\dots,e_{2p})$,
by $A_1$ the matrix in the basis $(f_1,\dots,f_{2k})$ of the residual involution associated with $i$ and $W$ with respect to $u$,
and by $N$ the matrix in the basis $(f_1,\dots,f_{2k})$ of the residual endomorphism of
$(W \oplus u(W))^{\bot}$ associated with $u$ and $W$. Then, the matrix of $i$ in $\bfB$ equals
$\begin{bmatrix}
0_{2p} & [0] & A^{-1} \\
[0] & A_1 & [0] \\
A & [0] & 0_{2p}
\end{bmatrix}$, while the matrix of $u$ in $\bfB$ is of the form
$\begin{bmatrix}
0_{2p} & [0] & [?] \\
[0] & N & [?] \\
S & [?] & [?]
\end{bmatrix}$. The resulting endomorphism $(iu)_W$ of $W$ is represented by the matrix $A^{-1}S$ in the basis $(e_1,\dots,e_{2p})$, while
the symplectic transformation of $W^{\bot}/W$ induced by $iu$ is represented by $A_1N$ in some symplectic basis.

\subsection{The search for good Lagrangians}\label{section:Lagrangianfind}

For the space-pullback technique, it appears crucial to find large totally $s$-singular subspaces $W$
of even dimension such that the form $s_{u,W}$ is nondegenerate, and we will even seek to be in the situation where
$s_{u,W}$ is symmetric. It turns out that in several key cases $W$ can be chosen as a Lagrangian of $(V,s)$. To see this, we will consider indecomposable s-pairs.

\begin{lemma}\label{lemma:existLagragianI}
Let $(s,u)$ be an indecomposable s-pair of type I or II, with underlying vector space denoted by $V$.
Then there exists a Lagrangian $\mathcal{L}$ of $(V,s)$ such that $s_{u,\mathcal{L}}$ is symmetric and nondegenerate.
\end{lemma}

\begin{proof}
We know that $u$ is cyclic. We denote by $p$ its minimal polynomial: it is a palindromial of even degree, denoted by $2n$.
Remember also that $p$ has no root in $\{1,-1\}$.

The commutative $\F$-algebra $R:=\F[t]/(p)$ is naturally equipped with the involution $\alpha \mapsto \alpha^\star$ that takes the class $\lambda$ of $t$ to its inverse,
and we have a natural structure of $R$-module on $V$ attached to $u$, so that
$s(y,\alpha\,z)=s(\alpha^\star\,y,z)$ for all $y,z$ in $V$ and all $\alpha \in R$.

We denote by $H:=\{\alpha \in R : \; \alpha^\star=\alpha\}$ the $\F$-linear subspace of all \textbf{Hermitian} elements.
Since neither $1$ nor $-1$ is a root of $p$, the skew-Hermitian element $\lambda-\lambda^{-1}=\lambda^{-1}(\lambda-1)(\lambda+1)$
is invertible, and it follows that $R=H \oplus (\lambda-\lambda^{-1}) H$ and $\dim_\F H=\frac{1}{2}\dim_\F R=n$.

Now, we choose a vector $x$ of $V$ that is cyclic for $u$, so that $\alpha \in R \mapsto \alpha\,x \in V$
is a vector space isomorphism. We set $\mathcal{L}:=H\,x$, so that $\dim \mathcal{L}=\dim_\F H=n$.
For every $\alpha\in H$, we have $s(x,\alpha\,x)=s(\alpha\,x,x)=-s(x,\alpha\,x)$, and hence $s(x,\alpha\,x)=0$.
It follows that for all $h_1$ and $h_2$ in $H$, we have $s(h_1\,x,h_2\,x)=s(x,h_1 h_2 x)=0$ because $h_1h_2$ is Hermitian.
Hence $\mathcal{L}$ is a Lagrangian of $(V,s)$.

Next, remember that the skew-symmetric part of $s_{u,\mathcal{L}}$ reads $(y,z) \mapsto \frac{1}{2}\,s(y,(u+u^{-1})(z))$. Then, noting that $\lambda+\lambda^{-1}$
is Hermitian, the previous argument can be used once more to see that this bilinear map is zero
(because $h_1h_2(\lambda+\lambda^{-1})$ is Hermitian for all Hermitian elements $h_1$ and $h_2$). Hence
$s_{u,\mathcal{L}}$ is symmetric.

Finally, let $y \in u(\mathcal{L}) \cap \mathcal{L}$. Then there are Hermitian elements $h_1$ and $h_2$ of $R$ such that
$u(h_1\,x)=y=h_2\,x$. Hence $\lambda h_1=h_2$, and by taking the adjoint we get $\lambda^{-1} h_1=h_2$, and hence
$(\lambda-\lambda^{-1})\,h_1=0$. Since $\lambda-\lambda^{-1}$ is invertible this yields $h_1=0$, and hence $y=0$.
Therefore $\calL \cap u(\calL)=\{0\}$. By Remark \ref{remark:nondeg}, it follows that $s_{u,\mathcal{L}}$ is nondegenerate.
\end{proof}

\begin{lemma}\label{lemma:existLagragianII}
Let $(s,u)$ be an indecomposable s-pair of type III, with underlying vector space denoted by $V$.
Then there exists a Lagrangian $\mathcal{L}$ of $(V,s)$ such that $s_{u,\mathcal{L}}$ is symmetric and nondegenerate.
\end{lemma}

\begin{proof}
Replacing $u$ with $-u$ if necessary, we lose no generality in assuming that $u$ is cyclic with minimal polynomial $(t-1)^{2n}$
for some integer $n>0$. We take a cyclic vector $x$ for $u$.

We proceed as in the proof of Lemma \ref{lemma:existLagragianI}, by equipping the commutative $\F$-algebra $R:=\F[t]/(t-1)^{2n}$ with the involution that takes the class $\lambda$ of $t$ to its inverse, and by considering the linear subspace $H$ of all Hermitian elements.

To see that $H$ has dimension at least $n$, we note that
the family $\bigl((\lambda+\lambda^{-1})^k\bigr)_{0 \leq k<n}$, whose terms belong to $H$, is linearly independent over $\F$:
this easily follows from the linear independence of $(\lambda^k)_{-n<k<n}$ in $R$ over $\F$, which follows from the fact that
the class of $t$ modulo $(t-1)^{2n}$ has no non-zero annihilating polynomial of degree less than $2n$.
Then, we consider the subspace $\calL:=H\,x$, which has dimension at least $n$.

Just like in the proof of Lemma \ref{lemma:existLagragianI}, one sees that $\calL$ is totally $s$-singular and $s_{u,\calL}$ is symmetric.
Therefore, as $\dim \calL \geq n$, we deduce that $\calL$ is a Lagrangian.

It remains to see that $u(\mathcal{L}) \cap \mathcal{L}=\{0\}$.
So, let $h_1,h_2$ be Hermitian elements of $R$ such that $\lambda h_1\,x=h_2\,x$. Then, just like in the proof of Lemma \ref{lemma:existLagragianI}
we arrive at $(\lambda-\lambda^{-1})\,h_1=0$, i.e.\ $\lambda^{-1}(\lambda+1)(\lambda-1)\,h_1=0$.
As $\lambda+1$ is invertible, this yields $(\lambda-1)\,h_1=0$, and hence
$h_1=\alpha (\lambda-1)^{2n-1}$ for some $\alpha \in \F$. Then
$h_1=h_1^\star=-\lambda^{-(2n-1)} h_1$. Since $\lambda^{2n-1}+1$ is invertible (because $t^{2n-1}+1$ is relatively prime with $(t-1)^{2n}$)
we deduce that $h_1=0$. We conclude that $u(\mathcal{L}) \cap \mathcal{L}=\{0\}$, which by Remark
\ref{remark:nondeg} shows that $s_{u,\mathcal{L}}$ is nondegenerate.
\end{proof}

\begin{cor}\label{cor:existlagrangian}
Let $(s,u)$ be an s-pair such that $u$ has no Jordan cell of odd size for an eigenvalue in $\{\pm 1\}$.
Then there exists an $s$-Lagrangian $\mathcal{L}$ such that $s_{u,\mathcal{L}}$ is symmetric and nondegenerate.
\end{cor}

\begin{proof}
Decomposing $(s,u)$ into a direct sum of indecomposable cells, we can split the underlying vector space $V$ of $(s,u)$ into an orthogonal direct sum $V=V_1 \botoplus \cdots \botoplus V_n$
of pairwise $s$-orthogonal subspaces, all stable under $u$, and for which each resulting s-pair $(s_i,u_i)$ is a cell of type I, II or III
(note the absence of cells of type IV, thanks to the assumption on the Jordan cells of $u$).
So, for all $i \in \lcro 1,n\rcro$, it follows from Lemmas \ref{lemma:existLagragianI} and \ref{lemma:existLagragianII} that we can find a Lagrangian $\calL_i$ of $V_i$
such that $(s_i)_{u_i,\calL_i}$ is symmetric and nondegenerate.
And then it is clear that $\calL:=\calL_1 \oplus \cdots \oplus \calL_n$ has the required properties.
\end{proof}

Next, we will see that, in specific situations, it is even possible to adjust the equivalence type of the
nondegenerate symmetric bilinear form $s_{u,\calL}$.

\begin{lemma}\label{lemma:existlagrangianfixbI}
Let $\lambda \in \F \setminus \{0,1,-1\}$.
Let $(s,u)$ be an s-pair of dimension $2n$, with underlying vector space denoted by $V$, such that $u$ is annihilated by $(t-\lambda)(t-\lambda^{-1})$,

Let $b$ be a nondegenerate symmetric bilinear form of rank $n$.
Then there exists a Lagrangian $\calL$ of $V$ such that the bilinear form $s_{u,\calL}$ is equivalent to $b$.
\end{lemma}

\begin{proof}
We choose an orthogonal basis $(e_1,\dots,e_n)$ for $b$ and we set
$a_i:=b(e_i,e_i)$ for $i \in \lcro 1,n\rcro$.

Next, we split $(s,u)$ into an orthogonal direct sum of indecomposable s-pairs $(s_1,u_1),\dots,(s_n,u_n)$
of type II, so that each $u_i$ is cyclic with minimal polynomial $(t-\lambda)(t-\lambda^{-1})$.
Using the same line of reasoning as in Corollary \ref{cor:existlagrangian}, we see that it suffices
to find, for all $i \in \lcro 1,n\rcro$, a vector $x_i$ in the underlying vector space of $(s_i,u_i)$
such that $s(x_i,u(x_i))=a_i$ (and then we take $\calL:=\Vect(x_1,\dots,x_n)$).

Hence, we are reduced to the following situation: let $(s',u')$ be an s-pair in which $u'$
is cyclic with minimal polynomial $(t-\lambda)(t-\lambda^{-1})$, with underlying vector space denoted by $V'$,
and let $\alpha \in \F \setminus \{0\}$. We need to prove that there exists $x'\in V'$ such that $s'(x',u'(x'))=\alpha$.
Note that $V'=\Ker(u'-\lambda \id) \oplus \Ker(u'-\lambda^{-1}\id)$, and
$\calL:=\Ker(u'-\lambda \id)$ and $\calL':=\Ker(u'-\lambda^{-1}\id)$ are transverse Lagrangians, whence $u'$
is the symplectic extension of $\lambda \id_{\calL}$ to $\calL'$.
This proves that $(s',u')$ is isometric to every s-pair $(s'',u'')$ in which $u''$
is cyclic with minimal polynomial $(t-\lambda)(t-\lambda^{-1})$.
Now, setting $M:=\begin{bmatrix}
0 & -\alpha^{-1} \\
\alpha & \lambda+\lambda^{-1}
\end{bmatrix}$, we see that $M$ is $K_2$-symplectic and cyclic with minimal polynomial
$(t-\lambda)(t-\lambda^{-1})$. From the above, we deduce that $M$ represents $u'$ in some symplectic basis
$(x',y')$ of $V'$, and we conclude that $s'(x',u'(x'))=\alpha$. This completes the proof.
\end{proof}

\begin{lemma}\label{lemma:existlagrangianfixbII}
Assume that $|\F|=3$.
Let $(s,u)$ be an s-pair with dimension $2n$ such that $u$ is annihilated by $t^2+1$,
with underlying vector space denoted by $V$.
Let $b$ be a nondegenerate symmetric bilinear form of rank $n$.
Then there exists a Lagrangian $\calL$ of $V$ such that $s_{u,\calL}$ is equivalent to $b$.
\end{lemma}

\begin{proof}
Working like in the previous proof, it will suffice to prove that, if $(s,u)$
has dimension $2$, then for every $\alpha \in \F \setminus \{0\}$ there exists
$x \in V$ such that $s(x,u(x))=\alpha$.

To see this, we shall prove that the matrices $M$ of $\SL_2(\F)$ such that $M^2+I_2=0$
form a conjugacy class in $\SL_2(\F)$. Indeed, if this holds true then there are symplectic bases $(x,y)$ and $(x',y')$ of $V$
in which the respective matrices of $u$ are $K_2=\begin{bmatrix}
0 & 1 \\
-1 & 0
\end{bmatrix}$ and $-K_2$, and then we see that $s(x,u(x))=-1$ and $s(x',u(x'))=1$, which yields the claimed result.

Now, let $M \in \SL_2(\F)$ be such that $M^2+I_2=0$. Then $M$ is not a scalar multiple of $I_2$ (because $-1$ is not a square in $\F$),
and hence there exists $P \in \GL_2(\F)$ such that $M=PK_2P^{-1}$.
Noting that $Q:=\begin{bmatrix}
1 & -1 \\
1 & 1
\end{bmatrix}=I_2-K_2$ commutes with $K_2$ and has determinant $-1$, we see that we can replace $P$ with $PQ$ if necessary.
Hence, we can \emph{choose} $P \in \SL_2(\F)$, which proves that the matrices $M$ of $\SL_2(\F)$ such that $M^2+I_2=0$
form a conjugacy class in $\SL_2(\F)$.
\end{proof}

\subsection{An application: some $3$-reflectional symplectic transformations}

\begin{prop}\label{prop:3refl}
Let $(s,u)$ be an s-pair. Assume that $u$ is a symplectic extension of a cyclic automorphism $v$ whose
minimal polynomial is even and relatively prime with its reciprocal polynomial.
Then $u$ is $3$-reflectional in $\Sp(s)$.
\end{prop}

For the remainder of the present article, the only case we need is the one where $(s,u)$ has dimension $4$.
Yet, in the prospect of future work on finite fields, it will be very useful to have the general case.

\begin{proof}
Denote by $W$ the underlying vector space of $v$, and by $n$ its dimension.
Denote by $p$ the minimal polynomial of $v$. Note that $n$ is even because $p$ is even.

We shall obtain the result backwards, by starting from a specific $2$-reflectional element.
First of all, we use the result of Stenzel \cite{Stenzel} that states that
there exists a nondegenerate symmetric bilinear form $B$ on $\F^n$
together with a $B$-skewselfadjoint automorphism $v'$ of $\F^n$ that is cyclic with minimal polynomial $p$.

Assume first that $|\F|>3$, and choose $\lambda \in \F \setminus \{0,1,-1\}$.
We can choose an s-pair $(s',u')$ of dimension $2n$ such that $u'$ is a symplectic extension of $\lambda \id$.
Since $n$ is even, Corollary \ref{cor:Nielsencor} shows that $u'$ is $2$-reflectional in $\Sp(s')$.
Next, by Lemma \ref{lemma:existlagrangianfixbI} there exists an $s'$-Lagrangian $W'$ such that $s'_{u',W'}$
is equivalent to $B$. It follows that there exists an $s'_{u',W'}$-skewselfadjoint transformation
$v'$ of $W'$ that is similar to $v$, and hence is cyclic with minimal polynomial $p$.

From there, we can apply the space-pullback technique to the pair $(u',W')$. This yields a symplectic involution $i_1 \in \Sp(s')$
such that $i_1 u'$ stabilizes $W'$, with resulting endomorphism $v'$.
And since the assumptions yield that the characteristic polynomial of $v'$ is relatively prime with its reciprocal polynomial,
we deduce from Corollary \ref{lemma:split2} that $i_1 u'$ is a symplectic extension of $v'$. Hence
$(s',i_1 u')$ is isometric to $(s,u)$. And since $i_1 u'$ is $3$-reflectional in $\Sp(s')$, we
deduce that $u$ is $3$-reflectional in $\Sp(s)$.

\vskip 3mm
Assume finally that $|\F|=3$. Then we can use the same line of reasoning as in the previous case,
the only difference being that $u'$ is chosen as a symplectic extension of a quarter turn, i.e.\ of an automorphism $f$ such that $f^2=-\id$.
In this case, we simply replace the use of Lemma \ref{lemma:existlagrangianfixbI} by the one of Lemma \ref{lemma:existlagrangianfixbII}:
the key is that $f$ is $2$-reflectional in the corresponding general linear group, which is a consequence of Wonenburger's theorem because
$f$ is similar to its inverse (all its invariant factors equal $t^2+1$, which is a palindromial).
\end{proof}

\section{The induction for infinite fields}\label{section:induction}

We are finally ready to prove one of our main results:

\begin{prop}\label{prop:induction}
Let $\F$ be an infinite field, and $n \geq 8$ be an even integer.
Then $\ell_n(\F) \leq \max(4,\ell_{n-4}(\F))$.
\end{prop}

By the examples in Section \ref{section:example4}, it is clear that $\ell_n(\F) \geq 4$ for all even $n \geq 4$, so as a consequence
of Proposition \ref{prop:induction} we will find that
each sequence $(\ell_{4k}(\F))_{k \geq 1}$ and $(\ell_{4k+2}(\F))_{k \geq 1}$ is nonincreasing, yielding Theorem \ref{theo:nonincreasing}.

\begin{proof}
Of course, we can assume that $r:=\max(4,\ell_{n-4}(\F))$ is finite, otherwise the result to be proved is obvious.

Let $(s,u)$ be an s-pair with dimension $n$, with underlying vector space denoted by $V$.
The idea is to use the space-pullback technique with a $2$-dimensional linear subspace, if possible.
More precisely, we will seek such a totally $s$-singular subspace $P$ so that $s_{u,P}$
is symmetric and nondegenerate.

The results obtained in Section \ref{section:Lagrangianfind} will help us find such a subspace in most cases, but this requires a bit of extra work due to cells of type IV.
Indeed, by decomposing $u$ into indecomposable summands and recomposing if necessary, we
split $u=u_r \botoplus u_e$ where $u_r$ has no Jordan cell of odd size for an eigenvalue in $\{\pm 1\}$, and
$u_e$ is an orthogonal direct sum of cells of type IV. Denote by $V_r$ and $V_e$ the respective underlying vector spaces.
Setting $m:=\dim V_r$, we directly apply Corollary \ref{cor:existlagrangian} to get a Lagrangian
$W$ of $V_r$ such that $s_{u,W}$ is symmetric and nondegenerate.

If $m \geq 4$ then it suffices to extract a plane $P \subseteq W$ on which $s_{u,W}$ is regular (it suffices to take the span of the first two vectors of an $s_{u,W}$-orthogonal basis).
Assume now that $m=2$ and that $u_e$ is not an involution. Then $u_e-(u_e)^{-1}$ is non-zero, and
as the symmetric part of the bilinear form $x \mapsto s(x,u_e(x))$ is $(x,y) \mapsto \frac{1}{2} s(x,(u_e-u_e^{-1})(y))$,
we can find a vector $x \in V_e$ such that $s(x,u_e(x)) \neq 0$. And finally $P:=W \oplus \F x$ satisfies the conclusion.
In the remaining cases:
\begin{itemize}
\item Either $m=0$, and then $u$ is the orthogonal direct sum of cells of type IV, each one of which is directly $2$-reflectional by
Nielsen's theorem (combined with Wonenburger's \cite{Wonenburger} to see that every Jordan cell for an eigenvalue $\pm 1$ is $2$-reflectional in
the corresponding general linear group).
\item Or $m=2$ and $u_e$ is an involution. In that case, as $n \geq 8$ we can extract an $s$-regular subspace
$Q$ of $V_e$ of dimension $4$ that is stable under $u_e$, and then we can resplit $u=u_1 \botoplus u_2$, where $u_2$ is an involution of $Q$
and $u_1$ is a symplectic transformation of $Q^{\bot}$. In that case $u_1$ is $r$-reflectional in $\Sp(s_1)$
(where $s_1$ denotes the restriction of $s$ to $(Q^\bot)^2$), whereas $u_2$ is directly an involution,
whence $u$ is $r$-reflectional in $\Sp(s)$.
\end{itemize}
Now, we can discard these two special cases and assume that we have found the requested $2$-dimensional space $P$.
We will now see how to use the space-pullback technique to conclude from here.

Denote by $u'$ the residual symplectic transformation of $V':=(P\oplus u(P))^{\bot}$
associated with $u$ and $P$ in the space-pullback technique, and by $s'$ the restriction of $s$ to $(V')^2$.
As $\dim (P\oplus u(P))^{\bot}=n-4$, we know that $u'=j_1 \cdots j_r$ for involutions $j_1,\dots,j_r$ in $\Sp(s')$.

We can choose an arbitrary symplectic form $b$ on $P$. Denote by $v$
the endomorphism of $P$ such that $\forall (x,y)\in P^2, \; s_{u,P}(x,y)=b(x,v(y))$.
Then $v$ is an $s_{u,P}$-skewselfadjoint automorphism, and hence it is cyclic (it cannot be a scalar multiple of the identity because
it is $s_{u,P}$-skewselfadjoint) and its characteristic polynomial reads $t^2+\alpha$ for some $\alpha \in \F \setminus \{0\}$.
Hence, for all $\lambda \in \F \setminus \{0\}$, we find that $\lambda^{-1} b$ is a symplectic form on $P$, and
the automorphism $v'$ of $P$ such that $\forall (x,y)\in P^2, \; s_{u,P}(x,y)=\lambda^{-1} b(x,v'(y))$
equals $\lambda v$; moreover, the characteristic polynomial of $\lambda v$ equals $t^2+\lambda^2 \alpha$.
Now, as $\F \setminus \{0\}$ is infinite we can choose $\lambda$ such that all the following conditions hold:
\begin{itemize}
\item $t^2+\lambda^2\alpha$ has no common root in $\overline{\F}$ with the characteristic polynomial of $j_2 \cdots j_r$;
\item $t^2+\lambda^2\alpha$ has no root which is a fourth root of the unity.
\end{itemize}
With such a choice, we apply the space-pullback technique to $u$, for the subspace $P$,
the symplectic form $\lambda^{-1} b$ and the residual involution $j_1$, yielding an involution $i \in \Sp(s)$
such that:
\begin{itemize}
\item $(iu)_P=v'$ is cyclic with minimal polynomial $t^2+\lambda^2\alpha$;
\item The symplectic automorphism of $P^\bot/P$ induced by $iu$ is symplectically similar to $j_2\cdots j_r$, and hence is $(r-1)$-reflectional;
\item The characteristic polynomial of the latter is relatively prime with $t^2+\lambda^2\alpha$.
\end{itemize}
If $t^2+\lambda^2\alpha$ were not relatively prime with its reciprocal polynomial, it would have a root
$z$ such that $z^{-1}=\pm z$, and hence $z$ would be a fourth root of the unity, contradicting our assumptions.
Hence, combining the above with Corollary \ref{lemma:split2} we obtain that $iu=u_1 \botoplus u_2$, where:
\begin{itemize}
\item $u_1$ is a symplectic extension of a cyclic automorphism with minimal polynomial $t^2+\lambda^2 \alpha$;
\item $u_2$ is $(r-1)$-reflectional in the corresponding symplectic group.
\end{itemize}
Using Proposition \ref{prop:3refl}, we find that $u_1$ is $3$-reflectional in the corresponding symplectic group, and we conclude
that $iu$ is $(r-1)$-reflectional in $\Sp(s)$. Hence $u$ is $r$-reflectional in $\Sp(s)$.
\end{proof}

\section{Completing the case $n=4$}\label{section:n=4}

Here, we will complete the $4$-dimensional case.
In most situations, we will use the same technique (with far fewer details!) as in the proof of
Proposition \ref{prop:induction}.

\begin{prop}\label{prop:4general}
Assume that $\F$ has more than $9$ elements.
Let $(s,u)$ be a $4$-dimensional s-pair such that $u$ has no Jordan cell of odd size for an eigenvalue $\pm 1$.
Then $u$ is $4$-reflectional in $\Sp(s)$.
\end{prop}

\begin{proof}
By Corollary \ref{cor:existlagrangian}, we find a Lagrangian $\mathcal{L}$ such that $s_{u,\mathcal{L}}$
is symmetric and nondegenerate.
Noting that there are less fourth roots of the unity than non-zero squares in $\F$ (this is where the assumption $|\F|>9$ comes into play),
we can find, for every $\lambda \in \F \setminus \{0\}$, a scalar $\beta \in \F \setminus \{0\}$ such that $\lambda \beta^2$ is not a fourth root of the unity.

From there, we use the same line of reasoning as in the proof of Proposition \ref{prop:induction} to find an involution $i \in \Sp(s)$
such that $iu$ stabilizes $\mathcal{L}$ and the resulting endomorphism is cyclic with minimal polynomial
$t^2+\alpha$ for some $\alpha \in \F \setminus \{0\}$ that is not a fourth root of the unity.
Then $iu$ is a symplectic extension of $(iu)_{\calL}$,
and by Proposition \ref{prop:3refl} we find that $iu$ is $3$-reflectional in $\Sp(s)$. Hence $u$ is $4$-reflectional in $\Sp(s)$.
\end{proof}

Now, only very special cases remain to be studied in order to solve the problem in dimension $4$. We will need a trick from the theory of
products of quadratic elements (see \cite{dSPsumprodregular} for more general results), for which we will give an elementary proof:

\begin{lemma}\label{lemma:splitquadratic}
Let $\lambda \in \F \setminus \{0,1,-1\}$.
Let $v$ be a cyclic automorphism of a vector space $P$ of dimension $2$, with determinant $1$.
Then $v$ is the product of two cyclic endomorphisms $v_1$ and $v_2$ of $P$, both with minimal polynomial $(t-\lambda)(t-\lambda^{-1})$.
\end{lemma}

\begin{proof}
To start with, we note that for every basis $(x,y)$ of $P$, there exists a cyclic automorphism $w$ of $P$,
with minimal polynomial $(t-\lambda)(t-\lambda^{-1})$ and such that $w(x)=y$: it suffices to take the automorphism of $P$
whose matrix in the basis $(x,y)$ is $\begin{bmatrix}
0 & -1 \\
1 & \lambda+\lambda^{-1}
\end{bmatrix}$. Moreover, every endomorphism of $P$ with characteristic polynomial $(t-\lambda)(t-\lambda^{-1})$ is cyclic
(otherwise it would be a scalar multiple of the identity, and its characteristic polynomial would have a double root).

Now, since $v$ is cyclic we can find a vector $x \in P$ such that $(x,v(x))$ is a basis of $P$.
By the first step, we can find a cyclic automorphism $v_1$ of $P$ with minimal polynomial $(t-\lambda)(t-\lambda^{-1})$ and such that $v_1(\lambda x)=v(x)$.
Setting $v_2:=v_1^{-1}v$, we see that $\det(v_2)=1$ and that $\lambda$ is an eigenvalue of $v_2$, and hence
the characteristic polynomial of $v_2$ equals $(t-\lambda)(t-\lambda^{-1})$. Hence $v=v_1v_2$ and we have the
claimed decomposition.
\end{proof}

Now, we assume that $\F$ is infinite and we prove that $\ell_4(\F) \leq 4$.
If $u$ has no Jordan cell of odd size for an eigenvalue $\pm 1$, then we conclude
directly by Proposition \ref{prop:4general}. Assume now that $u$ has at least one Jordan cell of odd size for the eigenvalue $1$ (the case of $-1$
is obtained by taking $-u$ instead of $u$). Then $u$ has two such cells at least, all of size $1$, by the classification of indecomposable s-pairs.
And it follows that $u=u_1 \botoplus u_2$ where $u_2$ is the identity on a $2$-dimensional space $V_2$, and $u_1$ is defined on a space $V_1$ of dimension $2$.
If $u_1=\pm \id$ then $u$ is an involution and we conclude immediately.

Assume now that $u_1\neq \pm \id$, to the effect that $u_1$ is cyclic with determinant $1$.
Let us choose $\lambda \in \F \setminus \{0,1,-1\}$. Then, by Lemma \ref{lemma:splitquadratic}, we can split $u_1=v_1 w_1$ where both $v_1$ and $w_1$ are cyclic with minimal polynomial $(t-\lambda)(t-\lambda^{-1})$. Finally, we can choose a cyclic automorphism $v_2$ of $V_2$ with minimal polynomial $(t-\lambda)(t-\lambda^{-1})$.
Then, we split
$$u=(v_1 \botoplus v_2)\,(w_1 \botoplus (v_2)^{-1}),$$
and we note that both symplectic transformations $v_1 \botoplus v_2$ and $w_1 \botoplus (v_2)^{-1}$ are diagonalisable over $\F$
with minimal polynomial $(t-\lambda)(t-\lambda^{-1})$ and eigenspaces of dimension $2$.
Hence, by Corollary \ref{cor:Nielsencor},
both  $v_1 \botoplus v_2$ and $w_1 \botoplus (v_2)^{-1}$
are $2$-reflectional in $\Sp(s)$, and we conclude that $u$ is $4$-reflectional in $\Sp(s)$.
This completes the proof that $\ell_4(\F) \leq 4$.

From there, combining $\ell_4(\F) \leq 4$ with Proposition \ref{prop:induction} yields
$\ell_{4n}(\F) \leq 4$ for all $n \geq 1$, thereby proving the first part of point (a) in Theorem \ref{theo:main}.

\section{The case $n=6$}\label{section:n=6}

Here, we prove the first part of point (b) in Theorem \ref{theo:main}.
Using Proposition \ref{prop:induction}, one sees that it suffices to prove that $\ell_6(\F) \leq 5$
whenever $\F$ is infinite. Our method here does not rely upon the space-pullback technique.
Rather, we will identify the action of certain $2$-reflectional elements on Lagrangians.
As before, we will only assume that $\F$ is infinite when it is absolutely unavoidable.

\subsection{Fitting the symmetric form on a good Lagrangian}

\begin{lemma}[Lagrangian fit lemma]\label{lemma:lagrangianfit}
Let $(V,s)$ be a symplectic space, $\mathcal{L}$ be a Lagrangian of $V$, and
$\mathcal{L'}$ be a transverse Lagrangian. Let $u \in \Sp(s)$ be such that
$s_{u,\mathcal{L}}$ is nondegenerate. Then there exists $u' \in \Sp(s)$ that is conjugated to $u$ and such that
$s_{u',\mathcal{L}}=s_{u,\mathcal{L}}$ and $u'(\mathcal{L})=\mathcal{L'}$.
\end{lemma}

\begin{proof}
Note that $u(\mathcal{L})$ is transverse to $\mathcal{L}$ because $s_{u,\mathcal{L}}$ is nondegenerate.
We construct  $w \in \Sp(s)$ such that $w$ leaves every vector of $\calL$ invariant and $w(\mathcal{L}')=u(\calL)$.
To do this, we start from a basis $(e_1,\dots,e_n)$ of $\calL$, and we extend it, first into a symplectic basis
$(e_1,\dots,e_n,f_1,\dots,f_n)$ of $V$ such that $(f_1,\dots,f_n)$ is a basis of $\calL'$, and next into
a symplectic basis $(e_1,\dots,e_n,g_1,\dots,g_n)$ of $V$ such that $(g_1,\dots,g_n)$ is a basis of $u(\calL)$;
then we take $w$ as the automorphism of $V$ that maps $e_1,\dots,e_n,f_1,\dots,f_n$ respectively to $e_1,\dots,e_n,g_1,\dots,g_n$.

So, by taking $u_1:=w^{-1} u w$, we have $u_1(\mathcal{L})=\mathcal{L}'$
and
$$\forall (x,y)\in \calL^2, \quad s_{u_1,\mathcal{L}}(x,y)=s(w(x),u(w(y)))=s_{u,\mathcal{L}}(w(x),w(y))=s_{u,\mathcal{L}}(x,y).$$
\end{proof}

Next is our key lemma for the $6$-dimensional case:

\begin{lemma}\label{lemma:dim6lemma}
Let $(V,s)$ be a $6$-dimensional symplectic space.
Let $\mathcal{L}$ be a Lagrangian of $V$, and
$b$ be a nondegenerate symmetric bilinear form on $\mathcal{L}$.
Then there exists a cyclic automorphism $v$ of $\mathcal{L}$ and a symplectic transformation $u \in \Sp(s)$ such that:
\begin{enumerate}[(i)]
\item $\det v=1$ and the minimal polynomial of $v$ is relatively prime with its reciprocal polynomial;
\item $u$ is $2$-reflectional in $\Sp(s)$;
\item $\forall (x,y) \in \mathcal{L}^2, \; s\bigl(x,u(y)\bigr)=b\bigl(x,v(y)\bigr)$.
\end{enumerate}
\end{lemma}

Again, the importance of this lemma for future work on finite fields justifies that we prove it in full generality,
and in particular that we care about fields with $3$ elements (which require a substantial adaptation of the proof).

\begin{proof}
We take a basis $(e_1,e_2,e_3)$ of $\mathcal{L}$ in which the Gram matrix $D$ of $b$ is diagonal,
and if $\F$ is finite we can refine the choice of this basis so that the first and last entry of $D$ are
chosen arbitrarily in $\F \setminus \{0\}$.

Let $A \in \GL_3(\F)$ be such that the $6$-by-$6$ symplectic matrix
$$M:=\begin{bmatrix}
A & 0_3 \\
0_3 & A^\sharp
\end{bmatrix}$$
is $2$-reflectional in $\Sp_6(\F)$ (we will adjust $A$ later on).
Let $S \in \Mats_3(\F)$. Conjugating $M$ with
$\begin{bmatrix}
I_3 & 0_3 \\
S & I_3
\end{bmatrix} \in \Sp_6(\F)$, we deduce that
$$M':=\begin{bmatrix}
A & 0_3 \\
SA-A^\sharp S & A^\sharp
\end{bmatrix}$$
is $2$-reflectional in $\Sp_6(\F)$.

Extending $(e_1,e_2,e_3)$ into a symplectic basis $\bfB:=(e_1,e_2,e_3,f_1,f_2,f_3)$ of $V$,
and taking the symplectic transformation $u$ that is represented in $\bfB$ by $M'$,
we find that the bilinear form $s_{u,\mathcal{L}}$ has its Gram matrix in
$(e_1,e_2,e_3)$ equal to
$$E:=SA-A^\sharp S.$$
The automorphism $v$ of $\mathcal{L}$ such that $\forall (x,y)\in \calL^2, \; s(x,u(y))=b(x,v(y))$
is represented by the matrix $D^{-1}E$ in $(e_1,e_2,e_3)$, and we will see that
$A$ and $S$ can be wisely chosen so that $D^{-1} E$ is cyclic, with determinant $1$ and whose minimal polynomial
is relatively prime with its reciprocal polynomial.

Note that if a monic polynomial of degree $3$ and constant coefficient $-1$ is not relatively prime with its reciprocal polynomial, then
it must have a root in $\{\pm 1\}$: indeed if it has two distinct roots of the form $\theta,\theta^{-1}$ then the third root must be $1$
(because of the constant coefficient), otherwise one of its roots $\theta$ satisfies $\theta=\theta^{-1}$ and hence $\theta=\pm 1$.
So, we will essentially seek that the polynomial we obtain has neither $1$ nor $-1$ among its roots.
And so that we are sure that the matrix $D^{-1}E$ is cyclic, we will simply ensure that it has no multiple eigenvalue in $\overline{\F}$.

From there, we split the discussion into two cases, whether $|\F|>3$ or $|\F|=3$.
Assume first that $|\F|>3$.
We start from an arbitrary $\lambda \in \F \setminus \{0,-1,1\}$ (which we will adjust afterwards) and we take $A:=\begin{bmatrix}
1 & 0 & 0 \\
0 & \lambda^{-1} & 0 \\
0 & 0 & \lambda^{-1}
\end{bmatrix}$. By Corollary \ref{cor:Nielsencor}, $M$ is $2$-reflectional in $\Sp_6(\F)$.
Varying $S$, we find that $E$ can take any value of the form
$$\begin{bmatrix}
0 & (\lambda^{-1}-1)\, C^T \\
(1-\lambda)\,C & (\lambda^{-1} -\lambda)\, S_0
\end{bmatrix}$$
with $C \in \F^2$ and $S_0 \in \Mats_2(\F)$.
Noting that $\lambda^{-1}-1 \neq 0$ and $\lambda^{-1}-\lambda \neq 0$, we deduce in particular that $E$ can take any value of the form
$E_{\lambda,d,e}=\begin{bmatrix}
0 & 0 & e \\
0 & d & 0 \\
\lambda e & 0 & 0
\end{bmatrix}$
with $(d,e) \in (\F \setminus \{0\})^2$ and $\lambda \in \F \setminus \{0,1,-1\}$.
Fixing such a triple, and writing $D=\mathrm{Diag}(d_1,d_2,d_3)$, we have
$$D^{-1}E_{\lambda,d,e}=\begin{bmatrix}
0 & 0 & d_1^{-1} e \\
0 & d_2^{-1} d & 0 \\
\lambda d_3^{-1} e & 0 & 0
\end{bmatrix}.$$
Now, fixing $\alpha$ and $\beta$ in $\F \setminus \{0\}$ and $\theta$ in $\F \setminus \{\pm d_3^{-1}d_1,0\}$,
we can freely choose the parameters $\lambda$, $d$ and $e$ so that
$$D^{-1}E_{\lambda,d,e}=\begin{bmatrix}
0 & 0 & \alpha \\
0 & \beta & 0 \\
\theta \alpha & 0 & 0
\end{bmatrix}.$$
The characteristic polynomial $p$ of $D^{-1}E_{\lambda,d,e}$ is then equal to $(t-\beta)(t^2-\theta\,\alpha^2)$.
We will now see that $\alpha,\beta,\theta$ can be chosen so that $p(0)=-1$ and $p$ has only simple roots in $\overline{\F}$.

Assume now that we have chosen the triple $(\alpha,\beta,\theta)$ so that $\beta \alpha^2 \theta=-1$.
Then $p=(t-\beta)(t^2+\beta^{-1})$, and $p$ has a multiple root only if $\beta^2+\beta^{-1}=0$.
Hence, $p$ has three distinct roots in $\overline{\F}$, all different from $1$ and $-1$, if and only if $\beta^3 \neq -1$ and $\beta \neq 1$.
In order to conclude, it suffices to find $x \in \F \setminus \{0\}$ such that $x^2 \neq -\theta^{-1}$
and $x^6 \neq \theta^{-3}$, and then to take $\alpha:=x$ and $\beta:=-x^{-2}\theta^{-1}$.
If $\F$ is infinite, it is obvious that such an element exists, whatever the choice of $\theta$ in $\F \setminus \{\pm d_3^{-1} d_1,0\}$.

Assume finally that $\F$ is finite and $|\F|>3$. In that case,
remember that we can choose $d_1$ and $d_3$ at will in $\F \setminus \{0\}$.
So, these parameters will be adjusted afterwards. And here we choose $\alpha:=1$
and then $\theta \in \F \setminus \{0,-1\}$ such that $\theta^3 \neq 1$: This is always possible, for if
not then $|\F|-2 \leq 3$ and hence $|\F|=5$, yet in $\F_5$ the only root of $t^3-1$ is $1$ because the group
$\F_5^\times$ has order $4$. Next, one chooses $\beta:=-\theta^{-1}$. And finally, one chooses
the starting basis so that $d_1=\theta$ and $d_3 \in \F \setminus \{0,1,-1\}$.

Hence, a good choice of $\lambda$ and $S$ can be always be made, which completes our proof in the case where $|\F|>3$.

\vskip 3mm
We finish by dealing with the remaining case where $|\F|=3$.
Here, we will force the minimal polynomial of the resulting matrix $D^{-1}E$ to be $t^3-t-1$, a polynomial that is easily seen to be irreducible over $\F$
(it has no root in $\F$, obviously).
First, we note that we can safely replace $u$ with $-u$, because if a symplectic transformation is $2$-reflectional
then so is its opposite. So, without loss of generality we can assume that $D=I_3$. In that case, we take
$$A:=\begin{bmatrix}
1 & [0]_{1 \times 2} \\
[0]_{2 \times 1} & K_2
\end{bmatrix} \quad \text{where} \quad K_2:=\begin{bmatrix}
0 & 1 \\
-1 & 0
\end{bmatrix}.$$
We note that $A$ is similar to its inverse, so by Nielsen's theorem $M$ is $2$-reflectional in $\Sp_6(\F)$.

Note that $(K_2)^\sharp=K_2$.
Now, let us take an arbitrary $4$-tuple $(a,b,c,d) \in \F^4$. Setting $C:=\begin{bmatrix}
c \\
d
\end{bmatrix} \in \F^2$ and $S_0:=\begin{bmatrix}
a & b\\
b & -a
\end{bmatrix}$, one checks that $S_0K_2-K_2 S_0=\begin{bmatrix}
b & -a \\
-a & -b
\end{bmatrix}$. Then $S:=\begin{bmatrix}
0 & C^T \\
C & S_0
\end{bmatrix}$ is symmetric and
$$SA-A^\sharp S=\begin{bmatrix}
0 & C^T(K_2-I_2) \\
(I_2-K_2)\,C & S_0 K_2-K_2 S_0
\end{bmatrix}=
\begin{bmatrix}
0 & -c-d & c-d \\
c-d & b & -a \\
c+d & -a & -b
\end{bmatrix}.$$
One checks that the characteristic polynomial of $E:=SA-A^\sharp S$ equals
$t^3-(a^2+b^2)t-4acd+2b(c^2-d^2)$.
Hence, by taking $a=1$, $b=0$, $c=1$ and $d=1$, we can adjust the characteristic polynomial of $D^{-1}E$ to equal
$t^3-t-1$, which completes the proof.
\end{proof}

\subsection{Completing the case $n=6$}

We need one more classical lemma before we can conclude:

\begin{lemma}[See Proposition 3.7 in \cite{dSPinvol3}]\label{lemma:cyclic3}
Let $f$ be a cyclic automorphism of a vector space $V$, with determinant $\pm 1$.
Then $f$ is $3$-reflectional in $\GL(V)$.
\end{lemma}

By using Lemma \ref{lemma:dim6lemma}, we can immediately conclude in most cases:

\begin{cor}\label{cor:dim6}
Let $(s,u)$ be an s-pair of dimension $6$ such that $u$ has no Jordan cell of odd size for an eigenvalue in $\{\pm 1\}$.
Then $u$ is $5$-reflectional in $\Sp(s)$.
\end{cor}

\begin{proof}
By Corollary \ref{cor:existlagrangian}, we can find a Lagrangian $\mathcal{L}$ such that $s_{u,\mathcal{L}}$ is symmetric and nondegenerate.
By Lemma \ref{lemma:dim6lemma}, we obtain an automorphism $v$ of $\mathcal{L}$ and a $2$-reflectional $u' \in \Sp(s)$ such that:
\begin{enumerate}[(i)]
\item $\forall (x,y)\in \mathcal{L}^2, \; s_{u',\calL}(x,y)=s_{u,\calL}(x,v(y))$;
\item $v$ is cyclic and its minimal polynomial $p$ is relatively prime with its reciprocal polynomial and satisfies $p(0)=-1$.
\end{enumerate}
Using Lemma \ref{lemma:lagrangianfit}, we can further assume that $u'(\mathcal{L})=u(\mathcal{L})$ (as a conjugation will preserve the property of being
$2$-reflectional!).
Now, we consider $u'':=u^{-1} u'$, which stabilizes $\mathcal{L}$, and we deduce from point (i) that the restriction of $u''$ to $\mathcal{L}$
equals $v$. Since the characteristic polynomial of $v$ is relatively prime with its reciprocal polynomial, Corollary \ref{lemma:split2}
shows that $u''$ is the symplectic extension of $v$ to some Lagrangian $\mathcal{L}'$.
Finally, Lemma \ref{lemma:cyclic3} shows that $v$ is $3$-reflectional in $\GL(\mathcal{L})$. Hence
$u''=s_{\mathcal{L}'}(v)$ is $3$-reflectional in $\Sp(s)$, and we conclude that $u=u' (u'')^{-1}$ is $5$-reflectional in $\Sp(s)$.
\end{proof}

Now, we can complete the $6$-dimensional case for infinite fields.
Assume that $\F$ is infinite, and let $(s,u)$ be an s-pair of dimension $6$.
Consider a decomposition of $(s,u)$ into indecomposable cells. If none of these indecomposable cells is of type IV,
then Corollary  \ref{cor:dim6} shows that $u$ is $5$-reflectional in $\Sp(s)$.
Otherwise, there are two options:
\begin{itemize}
\item $(s,u)$ is a cell of type IV, in which case it directly follows from
Nielsen's theorem and from Wonenburger's theorem that $u$ is $2$-reflectional in $\Sp(s)$;
\item Or $(s,u)$ has at least one indecomposable cell of type IV and dimension $2$,
in which case we have a splitting $u =u_1 \botoplus u_2$ where $u_2=\pm \id$ is defined over a $2$-dimensional space,
and $u_1$ is a symplectic transformation of a $4$-dimensional space. Then, by the $4$-dimensional case
$u_1$ is $4$-reflectional in the corresponding symplectic group, and as $u_2$ is an involution we conclude that $u$ is $4$-reflectional in $\Sp(s)$.
\end{itemize}
Hence, in any case $u$ is $5$-reflectional in $\Sp(s)$. This shows that $\ell_6(\F) \leq 5$.

Using Proposition \ref{prop:induction}, we deduce that $\ell_{4n+2}(\F) \leq 5$ for all $n \geq 1$, which completes the proof of the first part of statement (ii)
in Theorem \ref{theo:main}.

\section{Special examples}\label{section:examples}

Here, we give several examples that demonstrate how optimal some of our results are.

\subsection{Examples of symplectic transformations that are not products of three symplectic involutions}\label{section:example4}

Here, we give a systematic example of a symplectic transformation that is not $3$-reflectional.
This will prove that $\ell_{2n}(\F) \geq 4$ for every integer $n \geq 2$, whatever the field $\F$ under consideration (with characteristic other than $2$), thereby completing
the proof of Theorem \ref{theo:main}.

\begin{prop}\label{prop:examplenot3}
Let $p \in \F[t]$ be a monic palindromial of degree $2$ with no root in $\{\pm 1\}$.
Let $(s,u)$ be an s-pair with minimal polynomial $(t-1)\,p(t)$ and for which $\Ker p(u)$ has dimension $2$.
Then $u$ is not $3$-reflectional in $\Sp(s)$.
\end{prop}

Note that s-pairs satisfying the conditions in this proposition actually exist whenever the dimension $n$ is at least $4$!
Indeed, it suffices to take the orthogonal direct sum of the identity on an $(n-2)$-dimensional vector space equipped with a symplectic form and of a cyclic endomorphism
with minimal polynomial $p:=t^2+1$ (whose underlying vector space is equipped with an arbitrary symplectic form).

Over all fields however, it can be proved that every $u \in \Sp(s)$ with minimal polynomial $(t-1) (t+1)^2$ and $\dim \Ker(u+\id)^2=2$ is $3$-reflectional.

Our proof of Proposition \ref{prop:examplenot3} will involve several tricks from the theory of products of quadratic elements (see \cite{dSPsumprodregular}).
The first trick is specific to the symplectic group in dimension $4$:

\begin{lemma}[Trace Trick]
Let $(s,u)$ be an s-pair of dimension $4$. Assume that, in $\Sp(s)$, the element $u$ is $3$-reflectional but not $2$-reflectional.
Then $\tr u=0$.
\end{lemma}

\begin{proof}
There is an involution $i \in \Sp(s)$ such that $v:=iu$ is $2$-reflectional in $\Sp(s)$. By Nielsen's theorem,
$v$ is a symplectic extension of an endomorphism $w$ (of a $2$-dimensional vector space) which is similar to its inverse. Then, either
$w$ is an involution or $w$ is cyclic and its minimal polynomial is a palindromial.
In the first case $u$ would be directly $2$-reflectional, so the second case occurs and it shows that $v+v^{-1}=\alpha \id$
for some $\alpha \in \F$. Now, we have $u=iv$ and we gather that
$$u^\star=v^\star i^\star=v^{-1} i=(\alpha \id-v)i=\alpha i-vi.$$
Yet $u$ and $u^\star$ are similar as endomorphisms of a vector space, and hence they have the same trace. Therefore,
$$\tr(u)=\tr(u^\star)=\alpha \tr(i)-\tr(vi)=\alpha \tr(i)-\tr(iv)=\alpha \tr(i)-\tr(u).$$
Hence $2\tr(u)=\alpha \tr(i)$, and $\tr(i)=0$
otherwise $i=\pm \id$ and $u$ would already be $2$-reflectional! Hence $\tr(u)=0$.
\end{proof}

\begin{lemma}[Commutation Lemma for quadratic elements, lemma 4.2 in \cite{dSPsumprodregular}]\label{lemma:commutation}
Let $a$ and $b$ be automorphisms of a vector space $\F$ that are annihilated by monic polynomials $p$ and $q$
of degree $2$ such that $p(0)q(0) \neq 0$. Then, for $u:=ab$, the elements $a$ and $b$
commute with $u+p(0)q(0)u^{-1}$.
\end{lemma}

For example: if $a$ and $b$ are involutions, then they commute with $(ab)+(ab)^{-1}$;
if $a$ is an involution and $b$ is annihilated by a monic polynomial with degree $2$ and constant coefficient $1$, then $a$ and $b$ commute with $(ab)-(ab)^{-1}$. And so on.

\begin{lemma}[Stabilization Lemma for products of two involutions]\label{lemma:stabilization}
Let $i$ and $j$ be involutions of a vector space. Set $u:=ij$. Then $i$ and $j$
stabilize $\Ker(u-\eta \id)^k$ and $\im(u-\eta \id)^k$ for every integer $k \geq 0$ and every $\eta=\pm 1$.
\end{lemma}

This lemma can be obtained as a special case of corollary 4.5 from \cite{dSPprodexceptional}, but we
give a self-contained proof because it is short.

\begin{proof}
Let $\eta \in \{\pm 1\}$.
We check that $(u-\eta\id)j=i-\eta j=j (u^{-1}-\eta \id)$.
Hence, for all $k \in \N$, we have $(u-\eta\id)^k j=j (u^{-1}-\eta \id)^k$,
and from there we easily obtain the claimed stabilizations because
$(u^{-1}-\eta \id)^k$ and $(u-\eta \id)^k$ have the same kernel and the same range
(we can obtain each one from the other one by multiplying with a power of $u$ multiplied with a power of $-1$).
\end{proof}

We are now ready to prove Proposition \ref{prop:examplenot3}.
Throughout the proof we denote by $E_\lambda(f)$ the eigenspace of an endomorphism $f$ for the eigenvalue $\lambda$,
and by $E_\lambda^c(f)$ the corresponding characteristic subspace.

The proof works by induction on the dimension $n$ of the underlying vector space $V$ of $(s,u)$.
The case $n =2$ is trivial as then $u \neq \pm \id$.

Next, the case $n=4$ is a direct consequence of the trace trick: indeed, it is clear from Nielsen's theorem that $u$
cannot be $2$-reflectional, so if it is $3$-reflectional then $\tr u=0$.
Yet, the conditions clearly yield $\tr u=2+\tr p$, where $\tr p$ denotes the sum of the roots of $p$ in $\overline{\F}$, counted with multiplicities,
and $\tr u=0$ only if $p=(t+1)^2$, contradicting the assumption that $p(-1) \neq 0$.

In the remainder of the proof, we assume that $n \geq 6$. We use a \emph{reductio ad absurdum}, and assume that
there exists an involution $i \in \Sp(s)$ such that $v:=iu$ equals $j_1j_2$ for involutions $j_1$ and $j_2$ in $\Sp(s)$.
Set
$$V_+:=E_1(u) \cap E_1(i) \quad \text{and} \quad V_-:=E_1(u) \cap E_{-1}(i).$$
Since $\dim E_1(u) \geq n-2$, we have $\dim V_+ \geq \dim E_1(i)-2$ and
$\dim V_- \geq \dim E_{-1}(i)-2$, whence
\begin{equation}\label{eq:dimensioninequality}
\dim V_+ +\dim V_- \geq n-4.
\end{equation}
Note further that $V_+ \subseteq E_1(v)$ and $V_- \subseteq E_{-1}(v)$.
We will also write $V_1:=V_+$ and $V_{-1}:=V_-$ when it is more convenient.

Now, we arrive at a critical statement:

\begin{claim}\label{claim:noH}
There is no non-zero linear subspace $H$ of either $V_+$ or $V_-$
that is stable under $j_1$ and that is either totally $s$-singular or $s$-regular.
\end{claim}

\begin{proof}
Assume on the contrary that such a subspace $H$ exists.
Noting that $H$ is obviously stable under $v$
and $i$, we get that it is stable under all $i,j_1,j_2$, and we distinguish between two cases.
\begin{itemize}
\item Assume first that $H$ is $s$-regular. Then $i,j_1,j_2$ all stabilize $H^{\bot}$, which is $s$-regular.
Noting that $u$ is the identity on $H$ (because $H$ is included in $V_+$ or $V_-$), we gather that
$u$ induces a symplectic transformation $u_{H^\bot}$ of $H^\bot$ that is the product of three symplectic involutions (the ones induced by $i$, $j_1$ and $j_2$ on
$H^{\bot}$) and satisfies all the properties from the current proposition. If $H \neq V$ then by induction we find a contradiction.
But if $H=V$ then $V=V_+$ or $V=V_-$, in which case $u=\id$, which is obviously false.

\item Assume now that $H$ is totally $s$-singular. Then $i,j_1,j_2$ all stabilize $H^{\bot}$
and induce symplectic transformations $\overline{i}$, $\overline{j_1}$ and $\overline{j_2}$ of $H^{\bot}/H$
such that $\overline{i}\,\overline{j_1}\,\overline{j_2}$ equals the symplectic transformation $\overline{u}$
induced by $u$. Note that $u$ is the identity on $H$, and hence $u$ induces the identity on $V/H^{\bot}$ by Remark \ref{remark:stablesingular}.
Noting that the assumptions on $u$ yield $\chi_u=(t-1)^{n-2}p$, we find
$\chi_{\overline{u}}=(t-1)^{n-2-2\dim H} p$, and as $p$ has no root in $\{\pm 1\}$ we must have $\dim \Ker p(\overline{u})=2$.
As $(t-1)\, p$ annihilates $\overline{u}$, we obtain a contradiction by induction.
\end{itemize}
Hence, in each case we have obtained a contradiction.
\end{proof}

Next, remember from Lemma \ref{lemma:stabilization} that both $j_1$ and $j_2$ stabilize $E_1(j_1j_2)$
and $E_{-1}(j_1j_2)$.
In particular, if $V_+=E_1^c(v)$ and $E_1^c(v) \neq \{0\}$, then $H:=V_+$ would satisfy the conditions of Claim \ref{claim:noH} (in particular, it is $s$-regular).
Likewise, if $V_-=E_{-1}^c(v)$ and $E_{-1}^c(v) \neq \{0\}$, then $H:=V_-$ would satisfy them.
Hence none of these two special cases holds.

Now, we examine the possible invariant factors of $v$. First of all, assume that $v$ is not triangularizable with eigenvalues in $\{-1,1\}$.
Then, by Nielsen's theorem, either $\chi_v$ is a multiple of $q^2$ for some irreducible monic palindromial of degree at least $2$,
or $\chi_v$ is a multiple of $(qq^\sharp)^2$ for some irreducible monic polynomial $q$ that is distinct from $t+1$ and $t-1$, and with $q \neq q^\sharp$.
In each one of these cases, we see that $\dim E_1^c(v)+\dim E_{-1}^c(v) \leq n-4$.
Combining this with \eqref{eq:dimensioninequality} leads to $V_+=E_1^c(v)$ and $V_-=E_{-1}^c(v)$ thanks to the inclusions
$V_+ \subseteq E_1(v) \subseteq E_1^c(v)$ and $V_- \subseteq E_{-1}(v) \subseteq E_{-1}^c(v)$.
But in that situation we must have $E_1^c(v)=\{0\}=E_{-1}^c(v)$ thanks to the cases we have discarded,
and we arrive at $n=4$, contradicting our assumptions.

Hence $v$ is triangularizable with all eigenvalues in $\{\pm 1\}$.
Now, for $\varepsilon=\pm 1$ and $k \geq 1$, denote by $n_{k,\varepsilon}$ the number of Jordan cells of $v$
of size $k$ with respect to the eigenvalue $\varepsilon$. By Nielsen's theorem, all these numbers are even.
Noting that $\dim E_\varepsilon(v)$ is the total number of Jordan cells of $v$ for the eigenvalue $\varepsilon$,
we obtain
\begin{align*}
\sum_{k=1}^{+\infty} (k-1)(n_{k,1}+n_{k,-1}) & = n-(\dim E_1(v)+\dim E_{-1}(v)) \\
& \leq 4-(\dim E_1(v)-\dim V_+)-(\dim E_{-1}(v)-\dim V_-),
\end{align*}
and in particular $n_{k,1}=0=n_{k,-1}$ for all $k \geq 4$. Hence, we can split the discussion into six cases.

\vskip 2mm
\noindent \textbf{Case 1.}
$v$ has only Jordan cells of size $1$ (for the eigenvalues $\pm 1$), but then it would be an involution, and $u=iv$
would be $2$-reflectional: this would contradict Nielsen's theorem (which would yield that the dimension of $\Ker p(u)$ must be a multiple of $4$).

\vskip 2mm
\noindent \textbf{Case 2.} There is an $\varepsilon =\pm 1$ for which $n_{3,\varepsilon}=2$, and then $n_{k,\eta}=0$ for every other pair $(k,\eta)$ with $k \geq 2$.
Then $V_\varepsilon=E_\varepsilon(v)$, and hence $H:=\im (v-\varepsilon \id)^2 \cap E_\varepsilon(v)$ is included in $V_\varepsilon$.
Besides, $H$ is stable under $j_1$ and $j_2$ by Lemma \ref{lemma:stabilization}, and finally $H$ has dimension $2$ and is totally $s$-singular.
This contradicts Claim \ref{claim:noH}.

\vskip 2mm
\noindent \textbf{Case 3.} There exists $\varepsilon =\pm 1$ such that $n_{2,\varepsilon}=2$ and
$V_\varepsilon=E_\varepsilon(v)$. Then $n_{k,\varepsilon}=0$ for all $k \geq 3$.
Hence $H:=\im (v-\varepsilon \id) \cap E_\varepsilon(v)$ is included in $V_\varepsilon$,
is stable under $j_1$ and $j_2$, has dimension $2$ and is totally $s$-singular, and again this contradicts Claim \ref{claim:noH}.

\vskip 2mm
\noindent \textbf{Case 4.} There exists $\varepsilon =\pm 1$ such that $n_{2,\varepsilon}= 4$.
Then $V_\varepsilon=E_\varepsilon(v)$ and $n_{k,\varepsilon}=0$ for all $k \geq 3$,
and we proceed as in Case 3 with $H:=\im (v-\varepsilon \id) \cap E_\varepsilon(v)$ (which, this time around, has dimension $4$).

\vskip 2mm
\noindent \textbf{Case 5.} There exists $\varepsilon =\pm 1$ such that $n_{2,\varepsilon}=2$, $n_{2,-\varepsilon}=0$
and $V_{-\varepsilon}=E_{-\varepsilon}(v)$. Without loss of generality, we can assume that $\varepsilon=1$. Then
$n_{k,-1}=0$ for all $k \geq 2$, and hence $V_-=E_{-1}(v)=E_{-1}^c(v)$. Then $E_{-1}^c(v)=\{0\}$ (thanks to the discarding of special cases), and we deduce that $(v-\id)^2=0$.
By Lemma \ref{lemma:commutation}, we find that both $i$ and $v$ commute with $u-u^{-1}$. Yet,
$\Ker(u-u^{-1})=E_1(u)\oplus E_{-1}(u)=E_1(u)$, and it follows that $E_1(u)$ is stable under both $i$ and $v$.
Hence its orthogonal complement $\Ker p(u)$, which is $s$-regular, is also stable under $i$ and $v$. And since $\Ker p(u)$
has dimension $2$ the involution induced by $i$ on it must be $\pm \id$, leading to the contradiction that
the restriction of $u$ to $\Ker p(u)$ must be annihilated by $(t+1)^2$ or $(t-1)^2$.

\vskip 2mm
\noindent \textbf{Case 6.} There exists $\varepsilon =\pm 1$ such that $n_{2,\varepsilon}=2$, $n_{2,-\varepsilon}=0$
and $\dim V_+<\dim E_1(v)$ and $\dim V_-< \dim E_{-1}(v)$.
Then $\dim V_+=\dim E_1(v)-1$ and $\dim V_-=\dim E_{-1}(v)-1$.
Here, none of the previous methods seems to work, and we will go through a slightly different path instead.
Note that $E_1(v)$ is stable under $j_1$. If one of the eigenspaces of the resulting endomorphism has non-zero intersection
with $V_+$, then we take $x$ as a non-zero vector in this intersection, and we set $H:=\F x$;
 then $H$ is totally $s$-singular and stable under $i$, $v$ and $j_1$, so it satisfies our requirements.
 Likewise if one of the eigenspaces of the endomorphism of $E_{-1}(v)$ induced by $j_1$ has non-zero intersection
with $V_-$. By systematically taking an eigenspace of $j_1$ with maximal dimension, we deduce that $\dim V_+ \leq 1$ and $\dim V_-\leq 1$.
But then $n=6$ and we find a contradiction by coming back to the construction of $V_+$ and $V_-$: indeed,
as $i$ is a symplectic involution of $V$ one of its eigenspaces has dimension at least $4$, leading to $\dim V_+ \geq 2$ or $\dim V_- \geq 2$.

This final contradiction completes the proof of Proposition \ref{prop:examplenot3}.

\subsection{An element of $\Sp_4(\F_3)$ that is not $4$-reflectional}\label{section:F3}

Here, we consider the field $\F=\F_3$.
We shall consider s-pairs $(s,u)$ in which $u$ is cyclic with minimal polynomial $(t^2+1)(t-\eta)^2$ for some $\eta=\pm 1$.
Such pairs actually exist: we can obtain them as orthogonal direct sums $(s_1,u_1) \botoplus (s_2,u_2)$
where $u_1$ is cyclic with minimal polynomial $t^2+1$, and $u_2$ is cyclic with minimal polynomial $(t-\eta)^2$.
And we shall prove that for every such s-pair, $u$ is not $4$-reflectional in $\Sp(s)$.

We start with a basic result on $2$-reflectional elements.

\begin{lemma}\label{lemma:Nielsen4}
Assume that $|\F|=3$.
Let $(s,u)$ be a $4$-dimensional s-pair such that $u$ is $2$-reflectional in $\Sp(s)$, but $u$ is not an involution in $\Sp(s)$.
Then:
\begin{enumerate}[(i)]
\item $u$ is annihilated by a polynomial of degree $2$ and constant coefficient $1$;
\item Assume that we have a splitting $(s,u)=(s_1,u_1) \botoplus (s_2,u_2)$ into the orthogonal direct sum of two $2$-dimensional s-pairs.
Then $(s_2,u_2)$ is isometric to $(s_1,u_1^{-1})$.
\end{enumerate}
\end{lemma}

\begin{proof}
For the first point, note that $u$ is a symplectic extension of an automorphism $v$ of a $2$-dimensional vector space
such that $v^{-1}$ is similar to $v$. Then $\chi_v(0)=\pm 1$, and if $\chi_v(0)=-1$ then $v$ is an involution and hence so is $u$.
Therefore $\chi_v(0)=1$, and $v$ is annihilated by the palindromial $\chi_v$, so $u$ is also annihilated by it (because so is $(v^t)^{-1}$).
This proves point (i).

Next, if one of the $u_i$'s is not a scalar multiple of the identity, then it is necessary, by Nielsen's theorem, that the other one
is also not a scalar multiple of the identity, and even that it has the same characteristic polynomial.
Hence, if one of the $u_i$'s were equal to $\pm \id$, then so would be the other one, and $u$ would be an involution.

Hence $u_1$ and $u_2$ are cyclic with the same characteristic polynomial, which has degree $2$ and constant coefficient $1$.
If this characteristic polynomial is $t^2+1$, then we note that $u_1^{-1}$ is also cyclic with characteristic polynomial $t^2+1$, and
we go back to the proof of Lemma \ref{lemma:existlagrangianfixbII} to obtain that $(s_2,u_2) \simeq (s_1,u_1^{-1})$.

Assume now that the common characteristic polynomial of $u_1$ and $u_2$ is $(t-\eta)^2$. Taking $-u$ instead of $u$ if necessary, we can assume that $\eta=1$.
In a well-chosen $s_1$-symplectic basis and a well-chosen $s_2$-symplectic basis, the matrices of $u_1$ and $u_2$ read, respectively,
$A_1=\begin{bmatrix}
1 & -\varepsilon_1 \\
0 & 1
\end{bmatrix}$ and
$A_2=\begin{bmatrix}
1 & -\varepsilon_2 \\
0 & 1
\end{bmatrix}$.
Hence, $Q_u : x \mapsto s(x,u(x))$ is the orthogonal direct sums of two quadratic forms that are represented, in respective bases, by the matrices
$\begin{bmatrix}
0 & 0 \\
0 & \varepsilon_1
\end{bmatrix}$ and
$\begin{bmatrix}
0 & 0 \\
0 & \varepsilon_2
\end{bmatrix}$. It follows that the regular part of $Q_u$ is represented by the diagonal matrix
$\begin{bmatrix}
\varepsilon_1 & 0 \\
0 & \varepsilon_2
\end{bmatrix}$. Yet, by the remark that follows Lemma \ref{lemma:extensiontoproduct}, the regular part of $Q_u$ is hyperbolic because $u$ is $2$-reflectional
(again, this uses Nielsen's theorem). Hence
$\varepsilon_2=-\varepsilon_1$ (remember that $|\F|=3$ and hence the sole non-zero square in $\F$ is $1$).
Therefore $A_2=A_1^{-1}$, which yields the second point.
\end{proof}

Now, we are ready to prove our counterexample:

\begin{prop}\label{prop:F3not4}
Assume that $|\F|=3$.
Let $(s,u)$ be an s-pair in which $u$ is cyclic with minimal polynomial $(t^2+1)(t-\eta)^2$ for some $\eta=\pm 1$.
Then $u$ is not $4$-reflectional in $\Sp(s)$.
\end{prop}

\begin{proof}
Assume on the contrary that $u=u_1u_2$ where $u_1$ and $u_2$ are $2$-reflectional elements of $\Sp(s)$.
Noting that $\tr u=0+2\eta \neq 0$, we deduce from the Trace Trick that $u$ is not $3$-reflectional in $\Sp(s)$.
Hence, none of $u_1$ and $u_2$ is an involution. By point (i) of Lemma \ref{lemma:Nielsen4}, each $u_i$ is annihilated by a monic polynomial of degree $2$
and constant coefficient $1$. And in turn the Commutation Lemma (Lemma \ref{lemma:commutation}) shows that $u_1$ and $u_2$ commute with $u+u^{-1}$.

In particular, $u_1$ and $u_2$ stabilize the characteristic subspaces $P_0:=\Ker(u^2+\id)=\Ker(u+u^{-1})$ and
$P_\eta:=\Ker(u-\eta\id)^2=\Ker(u+u^{-1}-2\eta \id)$, both of which are $s$-regular and $2$-dimensional. Let $i \in \{1,2\}$.
Denote by $u_i^{(0)}$ and $u_i^{(\eta)}$ the resulting endomorphisms
of $u_i$ on $P_0$ and $P_\eta$, respectively, and by $s_0$ and $s_\eta$ the induced symplectic forms on $P_0$ and $P_\eta$.
By point (ii) of Lemma \ref{lemma:Nielsen4}, we find that $(s_0, u_i^{(0)}) \simeq \bigl(s_\eta, (u_i^{(\eta)})^{-1}\bigr)$.

Denote by $A_1,A_2$ the respective matrices of $u_1^{(0)}$ and $u_2^{(0)}$ in a given symplectic basis of $P_0$, and by $B_1,B_2$ the respective matrices of $u_1^{(\eta)}$ and $u_2^{(\eta)}$ in a given symplectic basis of $P_\eta$. Hence $B_1$ is conjugated to $(A_1)^{-1}$ in $\SL_2(\F_3)$ and $B_2$ is conjugated to $(A_2)^{-1}$ in $\SL_2(\F_3)$.
And we obtain
$$A_1A_2=K
\quad \text{and} \quad
B_1B_2=B,$$
where $K$ is cyclic with minimal polynomial $t^2+1$, and $B$ is cyclic with minimal polynomial $(t-\eta)^2$.

We will conclude thanks to the peculiarities of the group $\SL_2(\F_3)$.
Remember that the natural action of $\SL_2(\F_3)$ on the set $X$ of all $1$-dimensional linear subspaces of $(\F_3)^2$ yields an injective morphism from the quotient group $\SL_2(\F_3)/\{\pm I_2\}$ to the symmetric group $\mathfrak{S}(X)$, whose range is the alternating group $\mathfrak{A}(X)$. And remember that by taking the Klein subgroup $K(X)$ of
the latter, we can further map onto the (Abelian) quotient $\mathfrak{A}(X)/K(X) \simeq \Z/3$, yielding a surjective group homomorphism
$\pi : \SL_2(\F_3) \twoheadrightarrow \Z/3$.
So, with the previous identities we would have
\begin{multline*}
\pi(K)=\pi(A_1)+\pi(A_2)=\pi(B_1^{-1})+\pi(B_2^{-1})=-\pi(B_1)-\pi(B_2) \\
 =-\pi(B_1B_2)=-\pi(B).
\end{multline*}
Yet $K$ has order $4$, which is relatively prime with $3$, and hence $\pi(K)=0$. And conversely, since $\Ker \pi$ has order $8$, which is relatively prime with $3$,
it does not contain $B$, which has order $3$ or $6$ (whether $\eta=1$ or $\eta=-1$). This contradiction completes the proof.
\end{proof}

The previous proof can easily be adapted to yield the following result, which will be useful in future work on the topic:

\begin{lemma}\label{lemma:last}
Set $D:=\begin{bmatrix}
1 & 0 \\
0 & -1
\end{bmatrix} \in \Mat_2(\F_3)$. If the matrix $A:=\begin{bmatrix}
D & D \\
0 & D
\end{bmatrix}$ is not $3$-reflectional in $\Sp_4(\F_3)$, then it is not $4$-reflectional either.
\end{lemma}

\begin{proof}
We apply the same line of reasoning as in the proof of Proposition \ref{prop:F3not4}: assuming that $A$ is $4$-reflectional but not $3$-reflectional in $\Sp_4(\F_3)$,
we obtain two decompositions $J=A_1A_2$ and $-J=B_1B_2$ where $J=\begin{bmatrix}
1 & 1 \\
0 & 1
\end{bmatrix}$, and $B_i$ is conjugated to $A_i^{-1}$ in $\mathrm{SL}_2(\F_3)$ for all $i \in \{1,2\}$.
Then we write $\pi(J)=\pi(A_1)+\pi(A_2)=-\pi(B_1)-\pi(B_2)=-\pi(B_1B_2)=-\pi(J)$, leading to $\pi(J)=0$.
Yet $\pi(J) \neq 0$ because $J$ has order $3$ in $\mathrm{SL}_2(\F_3)$.
\end{proof}

In a subsequent article, we will prove that the matrix $A$ from Lemma \ref{lemma:last} is actually not $3$-reflectional in
 $\Sp_4(\F_3)$, but this is another story to be told.


\begin{thebibliography}{1}
\bibitem{Awa}
D. Awa, R.J. de La Cruz,
{Each real symplectic matrix is a product of symplectic involutions,}
Linear Algebra Appl.
{\bf 589} (2020) 85--95.

\bibitem{Djokovic}
D.\v{Z}. Djokovi\'c,
{Products of two involutions,}
Arch. Math. (Basel)
{\bf 18} (1967) 582--584.

\bibitem{EllersVilla}
E. Ellers, O. Villa,
{Generation of the symplectic group by involutions,}
Linear Algebra Appl.
{\bf 591} (2020) 154--159.


\bibitem{Gow}
R. Gow,
{Products of two involutions in classical groups of characteristic $2$,}
J. Algebra
{\bf 71} (1981) 583--591.

\bibitem{Gustafsonetal}
W.H. Gustafson, P.R. Halmos, H. Radjavi,
{Products of involutions,}
Linear Algebra Appl.
{\bf 13} (1976) 157--162.

\bibitem{HoffmanPaige}
F. Hoffman, E.C. Paige,
{Products of two involutions in the general linear group,}
Indiana Univ. Math. J.
{\bf 20} (1971) 1017--1020.

\bibitem{Hou1}
Z. Hou,
{Decomposition of symplectic matrices into products of commutators of symplectic involutions,}
Comm. Algebra
{\bf 48 (8)} (2020) 3459--3470

\bibitem{Hou2}
Z. Hou,
{Products of commutators of symplectic involutions,}
Linear Multilinear Algebra
{\bf 70 (15)} (2022) 2984--2997.

\bibitem{Hou3}
Z. Hou,
{Each real symplectic matrix is a product of commutators
of real symplectic involutions,}
Operators Matrices
{\bf 15 (4)} (2021) 1489--1504.

\bibitem{dLC}
R.J. de La Cruz,
{Each symplectic matrix is a product of four symplectic involutions,}
Linear Algebra Appl.
{\bf 466} (2015) 382--400.

\bibitem{Liu}
K.-M. Liu,
{Decomposition of matrices into three involutions,}
Linear Algebra Appl.
{\bf 111} (1988) 1--24.


\bibitem{dSPinvol3}
C. de Seguins Pazzis,
{Products of involutions in the stable general linear group,}
J. Algebra
{\bf 530} (2019) 235--289.


\bibitem{dSPorthogonal}
C. de Seguins Pazzis,
{Products of two involutions in orthogonal and symplectic groups,}
arXiv preprint, http://arxiv.org/abs/2210.03955

\bibitem{dSPunipotent}
C. de Seguins Pazzis,
{Products of unipotent elements of index $2$ in orthogonal and symplectic groups,}
arXiv preprint, http://arxiv.org/abs/2306.05821

\bibitem{dSPsumprodregular}
C. de Seguins Pazzis,
{The sum and the product of two quadratic matrices: regular cases,}
Adv. Appl. Clifford Algebras
{\bf 32} (2022).

\bibitem{dSPprodexceptional}
C. de Seguins Pazzis,
{The product of two invertible quadratic matrices: Exceptional cases,}
Linear Algebra Appl.
{\bf 662} (2023) 67--109.


\bibitem{Springer}
T.A. Springer,
{\"Uber Symplectische Transformatie,}
PhD Thesis, University of Leiden, 1951.

\bibitem{Stenzel}
H. Stenzel,
{\"{U}ber die Darstellbarkeit einer Matrix als Produkt von
zwei symmetrischer Matrizen, als Produkt von zwei alternierenden Matrizen
und als Produkt von einer symmetrischen und einer alternierenden Matrix,}
Math. Z.
{\bf 15} (1922) 1--25.

\bibitem{Wall}
G.E. Wall,
{On the conjugacy classes in orthogonal, symplectic and unitary groups,}
J. Austral. Math. Soc.
{\bf 3-1} (1963) 1--62.

\bibitem{Wonenburger}
M.J. Wonenburger,
{Transformations which are products of two involutions,}
J. Math. Mech.
{\bf 16} (1966) 327--338.

\end{thebibliography}
\end{document}